\documentclass{amsart}

\usepackage[all,cmtip]{xy}
\usepackage{mathrsfs}   
\usepackage{hyperref}
\usepackage{amsmath}
\usepackage{amssymb}
\usepackage{stmaryrd}

\theoremstyle{plain}
\newtheorem{theorem}{Theorem}[section]
\newtheorem{lemma}[theorem]{Lemma}
\newtheorem{proposition}[theorem]{Proposition}
\newtheorem{corollary}[theorem]{Corollary}
\newtheorem{conjecture}[theorem]{Conjecture}

\theoremstyle{definition}
\newtheorem{definition}[theorem]{Definition}

\theoremstyle{remark}
\newtheorem{remark}[theorem]{Remark}
\newtheorem*{claimu}{Claim}

\newcommand{\Z}{{\mathbb Z}}
\newcommand{\Q}{{\mathbb Q}}
\newcommand{\G}{{\mathbb G}}
\newcommand{\A}{{\mathbb A}}
\newcommand{\C}{{\mathbb C}}
\newcommand{\N}{{\mathbb N}}
\newcommand{\et}{\mathrm{\acute{e}t}}
\newcommand{\lcris}{\log\text{-}\mathrm{cris}}
\newcommand{\rig}{\mathrm{rig}}
\newcommand{\pn}{(\varphi,\nabla)}
\newcommand{\cal}[1]{\mathcal{#1}}
\newcommand{\rk}{\cal{R}_K}
\newcommand{\ekd}{\cal{E}_K^\dagger}
\newcommand{\pow}[1]{\llbracket #1 \rrbracket}
\newcommand{\lser}[1]{(\!(#1)\!)}

\newcommand{\isoc}[1]{\mathrm{Isoc}^\dagger(#1)}
\newcommand{\bu}{\bullet}
\newcommand{\spec}[1]{\mathrm{Spec}\left(#1\right)}
\newcommand{\spf}[1]{\mathrm{Spf}\left(#1\right)}
\newcommand{\norm}[1]{\left\vert#1\right\vert}
\newcommand{\isomto}{\overset{\sim}{\rightarrow}}
\newcommand{\isomfrom}{\overset{\sim}{\leftarrow}}
\usepackage[headheight=12.1pt]{geometry}
\begin{document}

\title{Around $\ell$-independence}
\author{Bruno Chiarellotto}
\email{chiarbru@math.unipd.it}
\address{Dipartimento di Matematica ``Tullio Levi-Civita''\\Universit\`a Degli Studi di Padova\\Via Trieste 63 \\35121 Padova\\Italia}
\author{Christopher Lazda}
\email{lazda@math.unipd.it}
\address{Dipartimento di Matematica ``Tullio Levi-Civita''\\Universit\`a Degli Studi di Padova\\Via Trieste 63 \\35121 Padova\\Italia}

\thanks{B. Chiarellotto was supported by the grants MIUR-PRIN 2015 "Number Theory and Arithmetic Geometry" and Padova PRAT CDPA159224/15. C. Lazda was supported by a Marie Curie fellowship of the Istituto Nazionale di Alta Matematica.}

\begin{abstract}
 In this article we study various forms of $\ell$-independence (including the case $\ell=p$) for the cohomology and fundamental groups of varieties over finite fields and equicharacteristic local fields. Our first result is a strong form of $\ell$-independence for the unipotent fundamental group of smooth and projective varieties over finite fields, by then proving a certain `spreading out' result we are able to deduce a much weaker form of $\ell$-independence for unipotent fundamental groups over equicharacteristic local fields, at least in the semistable case. In a similar vein, we can also use this to deduce $\ell$-independence results for the cohomology of smooth and proper varieties over equicharactersitic local fields from the well-known results on $\ell$-independence for smooth and proper varieties over finite fields. As another consequence of this `spreading out' result we are able to deduce the existence of a Clemens--Schmid exact sequence for formal semistable families. Finally, by deforming to characteristic $p$ we show a similar weak version of $\ell$-independence for the unipotent fundamental group of a semistable curve in mixed characteristic. \end{abstract}

\maketitle

\section{Introduction}

One intriguing aspect of the proof of the Weil conjectures by Grothendieck and his school, via the theory of \'etale cohomology, is in the need to choose an auxiliary prime $\ell$; the field $\Q_\ell$ of $\ell$-adic numbers then providing the coefficient field for a Weil cohomology theory. (To fix terms, let $F$ be a field, $X/F$ an algebraic variety, $F^\mathrm{sep}$ a separable closure and $G_F$ the corresponding absolute Galois group.) While \emph{a priori} it is not remotely clear how the cohomology groups $H^i_\et(X_{F^\mathrm{sep}},\Q_\ell)$ (considered as representations of $G_F$) are related for different values of $\ell\neq \mathrm{char}(F)$, the Riemann hypothesis \cite{Del74a} shows that at least over finite fields, they essentially contain the same information. Indeed, for curves $C$ (and again $\ell\neq \mathrm{char}(F)$) one essentially has $H^1_\et(C_{F^{\mathrm{sep}}},\Q_\ell) \cong (\mathrm{Jac}(C_{F^\mathrm{sep}})\otimes \Q_\ell )^\vee$ (again as Galois representations), $H^1$ being the only group of real interest in this case. 

To account for this plethora of different but subtly related cohomology groups, Grothendieck proposed the theory of pure motives, which for smooth and proper varieties at least should provide some form of algebro-geometric object $h^*(X)=\bigoplus_i h^i(X)$ such that $H^i_\et(X_{F^\mathrm{sep}},\Q_\ell)\cong h^i(X)\otimes \Q_\ell$ (again, suitably interpreted). Thus the groups $H^i_\et(X_{F^\mathrm{sep}},\Q_\ell)$ should be `independent of $\ell$', since all information they contain can be deduced from the `absolute' cohomology groups $h^i(X)$. In particular for curves, we essentially have $h^1(C)=\mathrm{Jac}(C_{F^\mathrm{sep}})^\vee$.

While the theory of motives, despite significant recent advances, is still incomplete, its prediction that various cohomology groups should be `independent of $\ell$' can, over certain ground fields, be precisely formulated in ways which do not depend on such a theory. For example, one such statement is that for smooth and proper varieties $X$ over finite fields of characteristic $p$, the trace of Frobenius on $H^i_\et(X_{F^\mathrm{sep}},\Q_\ell)$ (for $\ell\neq p$) has values in $\Z$ and is independent of $\ell$. Since the absolute Galois group of finite fields is generated by Frobenius, this tells us that, up to semisimplification, the various Galois representations $H^i_\et(X_{F^\mathrm{sep}},\Q_\ell)$ are, once restricted to $\mathrm{Frob}^{\Z}\subset G_F$, pairwise isomorphic (at least after base changing to any suitably large field $\Omega$).

Over more general fields, formulating similarly precise statements is somewhat tricky, since the Galois representation $H^i_\et(X_{F^\mathrm{sep}},\Q_\ell)$ depends on the interaction between the topology of $\Q_\ell$ and the profinite topology on $G_F$. For local fields (i.e. complete, discretely valued fields with finite residue fields) one can, thanks to the theory of Weil--Deligne representations and Grothendieck's $\ell$-adic monodromy theorem (and at least for $\ell$ different to the residue characteristic) package up the information contained in $H^i_\et(X_{F^\mathrm{sep}},\Q_\ell)$ in a way that only depends on $\Q_\ell$ as an abstract field, it therefore makes sense to conjecture that these groups become `pairwise isomorphic' over suitably large fields.

For $\ell=p$ one needs a slightly different approach. Over finite fields, one needs to replace \'etale cohomology by crystalline cohomology $H^i_\mathrm{cris}(X/K)$, and over mixed characteristic local fields one needs to replace the $\ell$-adic monodromy theorem by it's $p$-adic analogue, stating that the \'etale cohomology of varieties over mixed characteristic local fields is potentially semistable. Over finite fields, the formalism of crystalline cohomology (including weights) shows that the `independence of $\ell$' results can be extended to include the case $\ell=p$, and over mixed characteristic local fields it was shown in \cite{Fon94c} how to use $p$-adic Hodge theory to attach Weil--Deligne representations to de Rham Galois representations, thus again making sense of the $\ell$-independence conjectures `at $p$'. When $F$ is an equicharacteristic local field, then a more subtle use of crystalline cohomology allows one to view the cohomology of smooth and proper varieties over $F$ as $\pn$-modules over the Robba ring, and again the appropriate local monodromy theorem allows one to construct associated Weil--Deligne representations (see for example \cite[\S3]{Mar08}). Thus following Fontaine \cite{Fon94c} one can again formulate $\ell$-independence conjectures in this case, including at $\ell=p$. 

One can also ask about other expected `motivic' invariants, the most basic being the unipotent fundamental group. When $\ell\neq p$ then this can be defined as the $\Q_\ell$-pro-unipotent completion of the geometric \'etale fundamental group. When $F$ is finite and $\ell\neq p$ it is therefore acted on by Frobenius, and again we can ask whether or not this `non-abelian' representation of $\mathrm{Frob}^{\Z}$ is independent of $\ell$. When $F$ is local (and again when $\ell\neq p$) it is acted on `continuously' by $G_F$, so we can apply Grothendieck's monodromy theorem and produce a `non-abelian' Weil--Deligne representation. For $\ell=p$ and $F$ finite one again needs to take a different approach to define the pro-unipotent fundamental group, using Tannakian duality and the theory of isocrystals to produce a pro-unipotent group scheme on which  Frobenius acts. When $F$ is equicharacteristic local, one needs to appeal to the categories of overconvergent isocrystals introduced in \cite{LP16} and studied in \cite{Laz16b} in order to define the pro-unipotent fundamental group as a `non-abelian $\pn$-module' over the bounded Robba ring $\cal{E}^\dagger$. Again, appealing to the $p$-adic local monodromy theorem allows us to produce a non-abelian $p$-adic Weil--Deligne representations in this case. Hence in both these cases (i.e. $F$ finite or equicharacteristic local) one can formulate precise $\ell$-independence conjectures for the unipotent fundamental group, including the case $\ell=p$. When $F$ is mixed characteristic local and $\ell=p$ things are more tricky, since in order to apply Fontaine's methods one first needs to know that the $p$-adic unipotent $\pi_1$ is de Rham as a $G_F$-representation. For semistable curves this follows from the results of \cite{AIK15} and in general from results of D\'eglise and Nizio\l{} \cite{DN15}.

The purpose of this article, then, is to prove various cases of these $\ell$-independence conjectures, all including the case $\ell=p$. After recalling the various facts we need concerning cohomology and fundamental groups in \S\ref{recap}, the precise formulation of the conjectures is given in \S\ref{form}. In the rest of the article we will then prove certain versions of them in the following four cases.
\begin{enumerate} \item Unipotent fundamental groups of smooth projective varieties over finite fields.
\item Cohomology groups of smooth, proper varieties over equicharacteristic local fields.
\item Unipotent fundamental groups of smooth, proper varieties with semistable reduction over equicharacteristic local fields.
\item Unipotent fundamental groups of smooth, proper curves with semistable reduction over mixed characteristic local fields.
\end{enumerate}
In the first case, the subject of \S\ref{fgff}, we prove the strongest possible form of $\ell$-independence, stating that the various unipotent $\pi_1^\mathrm{uni}$'s become isomorphic as `non-abelian' Weil--Deligne representations over suitably large fields.  Here the basic idea behind the proof is formality, i.e. that the rational homotopy type is a formal consequence of the cohomology ring.

In the second case, treated in \S\ref{clf} and \S\ref{clf2}, we show a weaker form of $\ell$-independence, stating that the Frobenius seimsimplifications of the various cohomology groups become isomorphic over suitably large fields. The proof proceeds by first considering the case when the variety in question has semistable reduction (Theorem \ref{main2}). In this situation we use a `spreading out' result to reduce to the case of a global semistable family, and then apply a result of Deligne in \cite{Del80} stating that `$\ell$-independence' of a suitable local system on an open curve implies $\ell$-independence at the missing points. In fact Deligne's result only applies for $\ell\neq p$, but it is not difficult to follow through his proof `at $p$', as we do so in \S\ref{dp}. The general case is then deduced by using alterations and an argument involving weights (Theorem \ref{main22}).

In the third case, the topic of \S\ref{uel}, we prove a fairly weak form of $\ell$-independence, stating that the Frobenius semisimplifications of the graded pieces of the universal enveloping algebra of the Lie algebra of $\pi_1^\mathrm{uni}$ (for the filtration coming from the augmentation ideal) become isomorphic over suitably large fields. The proof here follows that of the second case, again reducing to the case of a global semistable family and applying the same independence result of Deligne. In \S\ref{cmc} we consider the last case, proving the same result for curves in mixed characteristic by reducing to the equicharacteristic situation.

After publication, it was pointed out to us by W. Zheng that the proof of Theorem \ref{main22} is invalid. In \S\ref{corrig} below we repeoduce a corrigendium to the original article, fixing the proof of Theorem \ref{main22}.

\subsection*{Acknowledgements}

Both authors would like to thank W. Nizio\l{} for suggesting the question of $\ell$-independence for unipotent fundamental groups over fields such as $k\lser{t}$, and C. Lazda would like to thank A. P\'al for first interesting him in $\ell$-independence questions in general, as well as for useful discussions related to the contents of this paper. We would also like to thank W. Zheng for pointing out the error in the proof of Theorem \ref{main22}, as well as L. Illusie for useful discussions concerning log structures, in particular the proof of Proposition \ref{prop: illusie-kato}. We would also like to thank the anonymous referee for a careful reading of an earlier version of the manuscript, in particular for correcting a mistake in our use of alterations.

\subsection*{Notations and conventions}

We will let $F$ denote either a finite field $k$ or a local field with residue field $k$ (usually equicharacteristic). As always, $p$ will be the characteristic of $k$, and we will let $q$ denote its cardinality. The term `Frobenius' without further qualification will always be understood to mean \emph{geometric} Frobenius, and Frobenius structures on isocrystals will, unless specified otherwise, mean $q$-power Frobenius structures. We will let $W=W(k)$ denote the ring of Witt vectors of $k$, $K$ its fraction field and $K^\mathrm{un}$ the maximal unramified extension of $K$. By `variety' we will mean `separated scheme of finite type'. We will denote a separable closure of $F$ by $F^{\mathrm{sep}}$, and the corresponding absolute Galois group by $G_F$. We will denote by $\cal{R}$ the Robba ring over $K$, and by $\underline{\mathbf{M}\Phi}^\nabla_{\cal{R}}$ the category of $\pn$-modules over $\cal{R}$. We will denote by $\ekd \subset \rk$ the bounded Robba ring.

\section{Preliminaries on fundamental groups and \texorpdfstring{$p$}{p}-adic cohomology}\label{recap}

The purpose of this section is to recall some general facts about unipotent fundamental groups, as well as some of the material from \cite{LP16} concerning $p$-adic cohomology over equicharacteristic local fields. So let $X/F$ be smooth, geometrically connected, quasi-projective variety and $x\in X(F)$. Then for any $\ell\neq \mathrm{char}(F)$ the $\ell$-adic version of the unipotent fundamental group has two different but equivalent interpretations. First of all we may consider the \'etale fundamental group $\pi_1^\et(X,\bar{x})$ based at a geometric point above $x$, together with the homotopy exact sequence
\[ 1 \rightarrow \pi_1^\et(X_{F^{\mathrm{sep}}},\bar{x})\rightarrow\pi_1^\et(X,\bar{x})\rightarrow  G_F \rightarrow 1.\]
The point $x$ induces a splitting of this exact sequence, and this gives rise to an action of $G_F$ on $\pi_1^\et(X_{F^{\mathrm{sep}}},\bar{x})$. Applying the Malcev completion functor over $\Q_\ell$ then gives a pro-unipotent group scheme $\pi_1^\et(X_{F^{\mathrm{sep}}},\bar{x})_{\Q_\ell}$ over $\Q_\ell$ which by functoriality is acted on by $G_F$.

Alternatively, we can use a Tannakian approach. That is, we consider the Tannakian category $\mathrm{Uni}_{\Q_\ell}(X_{F^{\mathrm{sep}}})$ of unipotent lisse $\Q_\ell$-sheaves on $X$. The point $x$ gives rise to a fibre functor $x^*:\mathrm{Uni}_{\Q_\ell}(X_{F^{\mathrm{sep}}})\rightarrow \mathrm{Vec}_{\Q_\ell}$ and hence we may define $\pi_1^\et(X_{F^{\mathrm{sep}}},\bar{x})_{\Q_\ell}$ to be the (pro-unipotent) affine group scheme representing automorphisms of $x^*$. By functoriality $G_F$ acts on $\mathrm{Uni}_{\Q_\ell}(X_{F^\mathrm{sep}})$, preserving $x^*$, and hence by Tannaka duality on $\pi_1^\et(X_{F^{\mathrm{sep}}},\bar{x})_{\Q_\ell}$. As the notation suggests, this coincides with the previous construction by \cite[Expos\'e VI, \S1.4.2]{SGA5}. 

While the first definition does not behave well when $\ell =\mathrm{char}(F)$, the second has an clear analogue, at least when $F=k$ is finite, replacing lisse $\Q_\ell$-sheaves on $X_{F^{\mathrm{sep}}}$ by overconvergent isocrystals on $X/K$, the category of which we will denote by $\mathrm{Isoc}^\dagger(X/K)$. Let $\mathrm{Uni}_K(X)$ denote its full subcategory of unipotent objects. Again, the point $x$ induces a fibre functor $x^*:\mathrm{Uni}_K(X)\rightarrow \mathrm{Vec}_K$ and by definition $\pi_1^\rig(X/K,x)$ is the (pro-unipotent) affine group scheme representating automorphisms of $x^*$. The Frobenius pullback functor $F^*:\mathrm{Uni}_K(X)\rightarrow \mathrm{Uni}_K(X)$ is an autoequivalence, and induces an isomorphism $F_*:\pi_1^\rig(X/K,x)\rightarrow \pi_1^\rig(X/K,x)$.

When $F=k\lser{t}$ is equicharacteristic local, we need to use the machinery of \cite{LP16} to define the $p$-adic unipotent fundamental groups. Since we will also need the results from \emph{loc. cit.} concerning $p$-adic cohomology for varieties over such $F$, we will take the opportunity now to quickly recap some of the main points from \cite{LP16}. There we constructed, for any variety $X/F$, cohomology groups $H^i_\rig(X/\ekd)$ as $\pn$-modules over the bounded Robba ring $\ekd$, satisfying all the expected properties of an `extended' Weil cohomology theory. In other words, we have finite dimensionality, vanishing in the expected degrees, versions with compact support or support in a closed subscheme, Poincar\'e duality, K\"{u}nneth formula \&c. We also constructed a category $\isoc{X/\ekd}$ of overconvergent isocrystals on $X$ relative to $\ekd$, and in \cite[Corollary 2.2]{Laz16b} it was proved that this category is Tannakian. The point $x$ provides a fibre functor
\[ x^*:\mathrm{Uni}_{\ekd}(X) \rightarrow \mathrm{Vec}_{\ekd} \]
from the subcategory of unipotent isocrystals, and again we may define $\pi_1^\rig(X/\ekd,x)$ to be the corresponding pro-unipotent group scheme over $\ekd$. It was also shown in \cite[\S5]{Laz16b} how to put a canonical `$\pn$-module structure' on $\pi_1^\rig(X/\ekd,x)$, that is a $\pn$-module structure on its Hopf algebra, which simply as a $\pn$-module (i.e. forgetting the Hopf algebra structure) is a direct limit of its finite dimensional sub-$\pn$-modules. 

Base changing the whole situation to $\rk$ we therefore obtain a non-abelian $\pn$-module
\[ \pi_1^\rig(X/\rk,x):=\pi_1^\rig(X/\ekd,x) \otimes_{\ekd} \rk \]
over $\rk$, as well as abelian ones
\[ H^i_\rig(X/\rk):=H^i_\rig(X/\ekd)\otimes_{\ekd} \rk .\]
These should be considered $p$-adic analogues of the $\ell$-adic Galois representations $\pi_1^\et(X_{F^{\mathrm{sep}}},\bar{x})_{\Q_\ell}$ and $H^i_\et(X_{F^\mathrm{sep}},\Q_\ell)$ respectively.

\section{The formalism of \texorpdfstring{$\ell$}{l}-independence, statement of the conjectures}\label{form}

In this section we will review some of the formalism behind the notion of $\ell$-independence over finite and equicharacteristic local fields, following \cite{Del80,Fon94c}. For this section, $F$ will stand for either a finite field $k$ or a local field with residue field $k$.

\begin{definition} We define the Weil group $W_F\subset G_F$ as follows.
\begin{itemize} \item If $F=k$ then we will let $W_F$ denote the subgroup of $G_F\cong \mathrm{Frob}_k^{\widehat\Z}$ consisting of \emph{integer} powers of Frobenius.
\item If $F$ is local, then we will let $W_F$ denote the inverse image of $W_k$ under the surjection $G_F\twoheadrightarrow G_k$.
\end{itemize}
It is topologised as follows: if $F$ is finite then $W_F$ is given the discrete topology, if $F$ is local then $W_F$ is given the unique topology such that the inertia group $I_F$ is open. 
\end{definition}

Since $W_F$ is pro-discrete, we may form the associated group scheme $W^{\mathrm{alg}}_F$ over $\Q$ by writing $W_F=\varprojlim_i G_i$ in the category of groups, viewing each discrete group $G_i$ as a group scheme over $\Q$ and then setting $W^\mathrm{alg}_F=\varprojlim_i G_i$ in the category of group schemes.

\begin{lemma} For any field $E$ of characteristic $0$ there is an equivalence of categories between continuous representations $\rho:W_F\rightarrow \mathrm{GL}(V)$ in finite dimensional $E$-vector spaces and finite dimensional algebraic representations $\rho^{\mathrm{alg}}:W_F^\mathrm{alg}\rightarrow \mathrm{GL}(V)$ of the group scheme $W_F^\mathrm{alg}$ over $E$.
\end{lemma}

If $F$ is local then there is a natural action of $W_F^\mathrm{alg}$ on the additive group $\G_a$ by setting 
\[ gxg^{-1} = q^{-v(g)}x\]
where $v(g)$ is such that $g$ maps to $\mathrm{Frob}_k^{v(g)}$ in $W_k$. 

\begin{definition} We define the Weil--Deligne group $W'_F$ to be $W_F^\mathrm{alg}$ if $F=k$ is finite or $W^\mathrm{alg}_F\ltimes \G_a$ if $F$ is local.
\end{definition}

Note that the usage of the term `Weil--Deligne' group to refer to $W_F^\mathrm{alg}$ over finite fields is not at all standard, and we employ it here only to be able to make uniform statements for finite and local fields.

We will now let $E$ be a field of characteristic $0$, and let $\cal{C}\rightarrow \mathrm{Aff}_E$ be a category fibred over the category of affine $E$-schemes. In other words for every $E$-algebra $R$ we have a category $\cal{C}_R$ and for every morphism $R\rightarrow R'$ we have a pullback (or `base extension') functor $\cal{C}_R\rightarrow \cal{C}_{R'}$. The examples the reader should keep in mind are the following.

\begin{itemize} \item the category of (quasi-)coherent $\cal{O}$-modules;
\item the category of pro-nilpotent Lie $\cal{O}$-algebras;
\item the category of (quasi-coherent) Hopf $\cal{O}$-algebras;
\item the category of pro-unipotent group schemes.
\end{itemize}
In particular, given an object $X\in \cal{C}_E$ it makes sense to speak of the functor on $E$-algebras
\[ \underline{\mathrm{Aut}}_E(X)(R)=\mathrm{Aut}_{\cal{C}_R} (X\otimes R ).\]

\begin{definition} A $\cal{C}_E$-representation of $W_F'$ is an object $X\in \cal{C}_E$ together with a morphism of functors $W_F'\rightarrow \underline{\mathrm{Aut}}_E(X)$. The category of such objects is denoted $\mathrm{Rep}_{\cal{C}_E}(W_F')$.
\end{definition}

Note that when $\cal{C}$ is the category of coherent $\cal{O}$-modules, then one recovers the usual notion of an $E$-valued Weil(--Deligne) representation, in general the point of this slightly fiddly approach is in order to have a well-defined notion for non-abelian objects such as Hopf algebras or group schemes. We can now finally make precise the concept of $\ell$-independence for (not necessarily abelian) Weil--Deligne representations.

\begin{definition}[\cite{Del73}, \S8] \begin{enumerate} \item Let $E'/E$ be a field extension, and $X\in \mathrm{Rep}_{\cal{C}_{E'}}(W_F')$ a $\cal{C}_{E'}$-valued Weil--Deligne representation. Then we say that $X$ is defined over $E$ if for any algebraically closed field $\Omega$ containing $E'$, $X \otimes \Omega\in \mathrm{Rep}_{\cal{C}_\Omega}(W_F')$ is isomorphic to all of its $\mathrm{Aut}(\Omega/E)$-conjugates.
\item Let $\{E_i\}_{i\in I}$ be a family of field extensions of $E$, and $\{X_i\}_{i\in I}$ a family of $\cal{C}_{E_i}$-valued Weil--Deligne representations. Then we say that $\{X_i \}_{i\in I}$ is \emph{$E$-compatible} if each $X_i$ is defined over $E$, and for any $i,j$ and any algebraically closed field $\Omega$ containing $E_i$ and $E_j$, the objects $X_i\otimes\Omega$ and $X_j\otimes\Omega$ in $\mathrm{Rep}_{\cal{C}_\Omega}(W_F')$ are isomorphic.
\end{enumerate}
\end{definition}

When $\cal{C}$ is the category of coherent $\cal{O}$-modules, then we also have a weaker notion of compatibility described as follows. If $F$ is finite then we define the `monodromy' filtration $M_\bu$ on a representation $V$ of $W_F'$ to be the trivial one, i.e. $0$ in negative degrees and $V$ otherwise. If $F$ is local then we will let $M_\bu$ be the usual monodromy filtration coming from viewing the $\G_a$ action as a nilpotent endomorphism $N:V\rightarrow V$. 

\begin{definition} Let $\{V_i\}_{i\in I}$ be a family of finite dimensional $E_i$-valued Weil--Deligne representations as above. We say that $\{V_i\}_{i\in I}$ is \emph{weakly $E$-compatible} if for all $k$ the character
$$ \mathrm{Tr}(-|\mathrm{Gr}_k^MV_i): W_F \rightarrow E_i
$$
of the $k$-th graded piece of the monodromy filtration has values in $E$ and is independent of $i$.
\end{definition}

Note that the difference between weak compatibility and compatibility is essentially that of Frobenius semisimplicity. For a representation $V$ of $W_F'$, we will let $V^{F\text{-}\mathrm{ss}}$ denote its Frobenius semisimplification. Thus when $F=k$ is finite this is just the semisimplification, and when $F$ is local we semisimplify the underlying $W_F$-representation.

\begin{lemma} \label{ws1} Let $\{V_i\}$ be a family of $W_F'$-representations. Then $\{V_i\}$ is weakly $E$-compatible if and only if $\{ V_i^{F\text{-}\mathrm{ss}}\}$ is $E$-compatible.
\end{lemma}

\begin{proof} When $F$ is finite this is clear, when $F$ is local this is \cite[Proposition 8.9]{Del73}.
\end{proof}

Now, if $X/F$ is a variety and $\ell\neq p$, then one can view the $i$th \'etale cohomology group $H^i_\et(X_{F^\mathrm{sep}},\Q_\ell)$ as a continuous representation of $G_F$, and hence as an $\Q_\ell$-valued algebraic representation of $W_F'$, when we consider it as the latter we will generally write $H^i_\ell(X)$. For $\ell=p$ and $F=k$ is finite, the natural Frobenius action on the rigid cohomology $H^i_\rig(X/K)$ allows us to consider it as a $K$-valued algebraic representation $H^i_p(X)$ of $W_F'$. For $\ell=p$ and $F\cong k\lser{t}$ equicharacteristic local, then we may consider the rigid cohomology groups $H^i_\rig(X/\rk)$ of $X$ as $\pn$-modules over the Robba ring $\rk$. Using Marmora's functor \cite{Mar08} from $\pn$-modules over $\rk$ to $K^\mathrm{un}$-valued Weil--Deligne representations, we may therefore consider the Weil--Deligne representation $H^i_p(X)$ associated to $H^i_\rig(X/\rk)$. When $\ell=p$ and $F$ is mixed characteristic local, then the \'etale cohomology $H^i_\et(X_{F^\mathrm{sep}},\Q_p)$ is de Rham as a representation of $G_F$, hence we may follow Fontaine's construction \cite{Fon94c} to produce a $K^\mathrm{un}$-valued Weil--Deligne representation, which we will again denote $H^i_p(X)$. Of course, we may also consider versions $H^i_{c,\ell}(X)$ with compact support. For cohomology, then, the $\ell$-independence conjectures (first formulated for $F$ local in mixed characteristic by Fontaine in \cite{Fon94c}) are as follows.

\begin{conjecture} \begin{itemize} \item $C_\ell(X,H^i)$: the system $\{ H^i_\ell(X) \}_{\ell}$ is $\Q$-compatible.
\item $C_{\ell,w}(X,H^i)$: the system $\{ H^i_\ell(X) \}_{\ell}$ is weakly $\Q$-compatible.
\item $C_\ell(X,H^i_c)$: the system $\{ H^i_{c,\ell} (X) \}_{\ell}$ is $\Q$-compatible.
\item $C_{\ell,w}(X,H^i_c)$: the system $\{ H^i_{c,\ell} (X) \}_{\ell}$ is weakly $\Q$-compatible.
\end{itemize} 
\end{conjecture}

When $F$ is finite, then $C_{\ell,w}(X,H^i)$ is known whenever $X$ is smooth and proper (for $\ell\neq p$ this follows from Deligne's proof of the Riemann hypothesis \cite{Del74a}, the fact that the same is true at $\ell=p$ then follows from \cite[Corollary 1]{KM74}). Even in this case, however, $C_{\ell}(X,H^i)$ is still wide open in general; by Lemma \ref{ws1} it would follow from the Frobenius semisimplicity conjecture. Hence for $X$ proper and smooth of dimension $d$ we know $C_\ell(X,H^i)$ for $i=0,1,2d-1,2d$, as well as for abelian varieties in all degrees. Let us also remark that in \cite{MO} the authors prove $C_{\ell,w}(X,H^i)$ and $C_{\ell,w}(X,H^i_c)$ whenever $\dim X\leq 2$.

When $F$ is local, then a much weaker version of $C_{\ell,w}(X,H^i)$ and $C_{\ell,w}(X,H^i_c)$ forms part of the main result of \cite{Zhe09}, however, even for smooth and projective varieties the fact that `purity' for $H^i$ is more complicated for local fields means that there is no straightforward deduction of $C_{\ell,w}(X,H^i)$ from Zheng's results. In \cite[Theorem 5.85]{LP16} it was proved that when $F\cong k\lser{t}$ is equicharacteristic local, then $C_{\ell,w}(X,H^i)$ holds for smooth (possibly open) curves and abelian varieties, or for smooth and proper varieties with good reduction. In fact, it was claimed there that in fact $C_{\ell}(X,H^i)$ holds for smooth curves and abelian varieties, however there is a gap in the proof which we shall correct in Lemma \ref{fssav} below.

Next let us turn to the unipotent fundamental group $\pi_1$, whose definition we recalled in the previous section. When $X/F$ is a geometrically connected variety with base point $x\in X(F)$, and $\ell\neq p$, then again we may consider the $\ell$-adic unipotent fundamental group $\pi_1^\et(X_{F^\mathrm{sep}},\bar{x})_{\Q_\ell}$ as a $G_F$-representation, we will write $\pi_1^{\ell}(X,x)$ for the associated $W_F'$-representation (with values in the fibred category of pro-unipotent group schemes). When $\mathrm{char}(F)=\ell=p$ we use the rigid fundamental group ($\pi_1^\rig(X/K,x)$ when $F=k$ is finite, $\pi_1^\rig(X/\rk,x)$ when $F\cong k\lser{t}$ is local) to produce a $p$-adic representation of $W_F'$, with values in the fibred category of pro-unipotent group schemes, which we will denote by $\pi_1^p(X,x)$ (again using Marmora's functor from $\pn$-modules to Weil--Deligne representations when $F\cong k\lser{t}$). When $\mathrm{char}(F)=0$, $\ell=p$ and $X$ is a semistable curve we note that by \cite[Theorem 1.8]{AIK15} the $p$-adic \'etale pro-unipotent fundamental group $\pi_1^\et(X_{F^\mathrm{sep}},\bar{x})_{\Q_p}$ is de Rham as a non-abelian representation of $G_F$, hence we may again associate to it a $K^\mathrm{un}$-valued (non-abelian) Weil--Deligne representation $\pi_1^p(X,x)$. For more general $X$, there are results of D\'eglise and Nizio\l{} in \cite{DN15} which again state that $\pi_1^\et(X_{F^\mathrm{sep}},\bar{x})_{\Q_p}$ is de Rham.

We will write $A_{\ell,X,x}$ for the Hopf algebra of $\pi_1^\ell(X,x)$, $L_{\ell,X,x}$ for its Lie algebra, $\hat{\cal U}_{\ell,X,x}$ for the completed universal enveloping algebra of $L_{\ell,X,x}$ and $\mathfrak{a}_{\ell,X,x}$ for the augmentation ideal of  $\hat{\cal U}_{\ell,X,x}$ (these are representations in Hopf algebras, pro-nilpotent Lie algebras and associative algebras respectively, and for each $k\geq1$ we may consider $\hat{\cal U}_{\ell,X,x}/\mathfrak{a}^k_{\ell,X,x}$ as a Weil--Deligne representation in vector spaces). The strongest form of $\ell$-independence one can state for fundamental groups is the following.

\begin{conjecture}[$C_{\ell}(X,\pi_1^\mathrm{uni})$] The collection $\{ \pi_1^\ell(X,x)\}_\ell$ of pro-unipotent groups with an action of $W_F'$ is $\Q$-compatible. Equivalently the collection $\{A_{\ell,X,x}\}_\ell$ (resp. $\{ L_{\ell,X,x}\}_\ell$, $\{ \hat{\cal U}_{\ell,X,x}\}_\ell$) of Hopf algebras (resp. Lie algebra, resp. associative algebras) is $\Q$-compatible.
\end{conjecture}

We will also be interested in the following weakening of the above conjecture.

\begin{conjecture}[$C_{\ell}(X,\hat{\cal U}/\mathfrak{a}^\bu),C_{\ell,w}(X,\hat{\cal U}/\mathfrak{a}^\bu)$] Fix $k\geq 1$. Then the system $\{ \hat{\cal U}_{\ell,X,x}/\mathfrak{a}^k_{\ell,X,x}\}_\ell$ of $W_F'$-representations is (weakly) $\Q$-compatible.
\end{conjecture}

\begin{remark} Note that in this conjecture we only require compatibility as $W_F'$-representations, that is we are saying nothing about the algebra structure. We clearly have the implications $C_{\ell}(X,\pi_1^\mathrm{uni})\Rightarrow  C_{\ell}(X,\hat{\cal U}/\mathfrak{a}^\bu)\Rightarrow  C_{\ell,w}(X,\hat{\cal U}/\mathfrak{a}^\bu)$.
\end{remark}

We would like to finish this section by fixing a hole in the proof of \cite[Theorem 5.85]{LP16}, stating that when $F\cong k\lser{t}$ is equicharacteristic local and $X/F$ is a smooth curve or an abelian variety, then $C_\ell(X,H^i)$ holds, in fact we only showed there that $C_{\ell,w}(X,H^i)$ holds, that is weak $\ell$-independence. Given Lemma \ref{ws1}, the next result completes the proof of `strong' $\ell$-independence in these cases.

\begin{lemma} \label{fssav} Let $F\cong k\lser{t}$ and $X/F$ be either a smooth curve or an abelian variety. Then $H^1_\ell(X)$ is Frobenius semisimple for all $\ell$. 
\end{lemma}

\begin{proof} First suppose $X=A$ is an abelian variety, the proof of \cite[Theorem 5.88]{LP16} shows that the graded pieces of the monodromy filtration on $H^1_\ell(A) $ (for any $\ell$, including $\ell=p$) are given by the following:
\begin{itemize} \item the cohomology $H^1_\ell(T)$ of a torus over $F$;
\item the cohomology $H^1_\ell(B)$ of an abelian variety with potentially good reduction;
\item the Weil--Deligne representation coming from a continuous representation $G_{F}\rightarrow \Z^s$ for some $s\geq0$.
\end{itemize}
Frobenius semisimplicity of the first and third are clear (since they both factor through a finite quotient of $G_F$), the second follows from Frobenius semisimplicity for abelian varieties over finite fields. Since the three graded pieces are of different weights, Frobenius semisimplicity of $H^1_\ell(A)$ follows. The deduction for smooth (possibly open) curves now follows exactly as in \cite[\S5.4]{LP16}.
\end{proof}

\section{Unipotent fundamental groups over finite fields}\label{fgff}

In this section, we will prove the following result.

\begin{theorem} \label{main1} Let $F=k$ be a finite field, and $X/F$ a smooth, projective variety. Then $C_\ell(X,\pi_1^\mathrm{uni})$ holds.
\end{theorem}

Let us first note the following consequence of the weak Lefschetz theorem.

\begin{proposition} To prove Theorem \ref{main1} it suffices to consider $X$ of dimension $\leq 2$. 
\end{proposition}

\begin{proof} If $\dim X>2$ then choose an iterated hyperplane section $Y\hookrightarrow X$ passing through $x$ with $\dim Y=2$. The weak Lefschetz theorem \cite[Expos\'e XII, Corollaire 3.5]{SGA2} says that the map $\pi_1^\et(Y_{F^\mathrm{sep}},\bar x)\rightarrow \pi_1^\et(X_{F^\mathrm{sep}},\bar x)$ is an isomorphism, hence the same is true of their $\Q_\ell$-unipotent completions, in other words $\pi_1^\ell(Y,x)\rightarrow \pi_1^\ell(X,x)$ is an isomorphism for $\ell\neq p$. 

When $\ell=p$ we use the weak Lefschetz theorem slightly differently. We claim that for any unipotent isocrystal $E$ on $X/K$, the induced map
\[ H^i_\rig(X/K,E)\rightarrow H^i_\rig(Y/K,E|_Y)\] 
is an isomorphism for $i=0,1$ and injective for $i=2$, this suffices to prove $\pi_1^p(Y,x)\rightarrow \pi_1^p(X,x)$ is an isomorphism by combining  Propositions 1.2.2 and 1.3.1 of \cite{CLS99b} with Tannakian duality. If $E$ is constant then this is simply the usual weak Lefschetz, in general one inducts on the unipotence degree and uses the five lemma. 
\end{proof}

The key ingredient in the proof of Theorem \ref{main1} is the following concrete description of the Lie algebra $\mathrm{Lie}\;\pi_1^\ell(X,x)$, due to Pridham \cite{Pri09}.  

\begin{proposition} Let $L(H^1_\ell(X)^\vee)$ denote the free Lie algebra on the dual of $H^1_\ell(X)$, and $I_\ell$ the ideal generated by the image of the dual of the cup product
\[ H^2_\ell(X)^\vee \overset{\cup^\vee_\ell}{\rightarrow} H^1_\ell(X)^\vee \otimes H^1_\ell(X)^\vee. \]
Then there is a $W_F$-equivariant isomorphism
\[ L_{\ell,X,x}=\mathrm{Lie}\;\pi_1^\ell(X,x) \cong \frac{L(H^1_\ell(X)^\vee)}{I_\ell}.\]
\end{proposition}

\begin{proof} For $\ell\neq p$ this is \cite[Corollary 2.15]{Pri09}, when $\ell=p$ exactly the same proof works, using the $p$-adic formalism of weights (see for example \cite{Ked06b}).
\end{proof}

Since Frobenius semisimplicity is known for $H^1_\ell(X)$ for $X$ smooth and projective of any dimension, and any $\ell$, we deduce the following corollary.

\begin{corollary} \label{main1red} To prove Theorem \ref{main1}, it suffices to show that the image of the cup product
\[ H^1_\ell(X) \otimes H^1_\ell(X) \overset{\cup_\ell}{\rightarrow} H^2_\ell (X)\]
satisfies weak $\ell$-independence, i.e. the family of representations $\{\mathrm{im}(\cup_\ell)\}_\ell$ is weakly $\Q$-compatible.
\end{corollary}

Clearly this holds when $X$ is a smooth, projective curve, since then the cup product map is surjective. When $X$ is a surface, we use the existence of K\"unneth projectors on $X$. 

\begin{proposition}[Corollary 2A10 and Lemma 2.4, \cite{Kle68}] Let $X/F$ be a smooth projective surface. Then there exist correspondences $\varpi_i\in \mathrm{CH}_2(X\times X)$ such that for all $\ell$, the action of $\varpi_i$ on cohomology $H_\ell^*(X)=\bigoplus_i H^i_\ell(X)$ is to project to the summand $H^i_\ell(X)$.
\end{proposition}

We can now complete the proof of Theorem \ref{main1}.

\begin{proof}[Proof of Theorem \ref{main1}] We may assume that $\dim X =2$. Choose $\varpi_1\in \mathrm{CH}_2(X\times X)$ as above, and let $\varpi_{1\otimes 1}:=\varpi_1\times \varpi_1 \in \mathrm{CH}_4(X^4)$. Then the effect of $\varpi_{1\otimes 1}$ on $H_\ell^*(X\times X)$ is to project to the summand $H^1_\ell(X)\otimes H^1_\ell(X)$.
 
Now let $\varpi_{1\cup 1}$ be the pullback of $\varpi_{1\otimes 1}$ to $\mathrm{CH}_2(X\times X)$ via $\Delta\times \Delta:X^2\rightarrow X^4$. The effect of $\varpi_{1\cup 1}$ on $H^*_\ell(X)$ is therefore to project onto the image of $\cup_\ell$ inside $H^2_\ell(X)$. Using Corollary \ref{main1red} the result now follows by composing $\varpi_{1\cup 1}$ with the graph of (some power of) Frobenius and applying the Lefschetz fixed point formula.
\end{proof}

\section{Cohomology over equicharacteristic local fields I}\label{clf}

The main result of this section is the following.

\begin{theorem}\label{main2} Suppose that $F\cong k\lser{t}$ is equicharacteristic local, and let $X/F$ be a smooth, proper variety with semistable reduction. Then $C_{\ell,w}(X,H^i)$ holds. 
\end{theorem}

This is the main step in showing such $\ell$-independence unconditionally for smooth and proper varieties over $F$. The first key result we need in the proof of Theorem \ref{main2} essentially allows us to reduce to the `globally defined' case, and in fact works in slightly larger generality than we will need. So for now let $S$ denote a Dedekind scheme, $s\in S$ a closed point and $\widehat{\cal{O}}_{S,s}$ the completed local ring at $s$. By definition a semistable scheme over $\widehat{\cal{O}}_{S,s}$ is one that is \'etale locally smooth over $\spec{\widehat{\cal{O}}_{S,s}[x_1,\ldots,x_r]/(x_1\cdots x_r)}$.

\begin{proposition} \label{spread} Let $\cal{X} \rightarrow \spec{\widehat{\cal{O}}_{S,s}}$ be a semistable scheme of finite type and $n\geq 2$ an integer. Then there exists an \'etale neighbourhood $(U,u)\rightarrow (S,s)$ of $s$ and a flat scheme $\cal{Y} \rightarrow U$, smooth away from $u$ and semistable at $u$, such that 
\[ \cal{Y} \times_U \frac{ \cal{O}_{U,u}}{\mathfrak{m}^n_u} \cong \cal{X} \times_{\widehat{\cal{O}}_{S,s}} \frac{ \widehat{\cal{O}}_{S,s}}{\mathfrak{m}^n_s} \]
as schemes over $ \cal{O}_{U,u}/\mathfrak{m}^n_u \cong \widehat{\cal{O}}_{S,s}/\mathfrak{m}^n_s$.
\end{proposition}

\begin{proof} Choose a local parameter $t$ at $s$. After shrinking $S$ we may assume that $S=\spec{R}$ is affine and $t\in R$. Then there exists an \'etale cover $\{ U_i\rightarrow \cal{X} \}$ of $\cal{X}$ such that each $U_i$ is \'etale over some $\spec{\widehat{\cal{O}}_{S,s}[x_1,\ldots,x_{d+1}]/(x_1\cdots x_r -t)}$ (with $r$ and $d$ allowed to vary with $i$.) 

Choose some finitely generated, normal $R$-algebra $A$ contained within $\widehat{\cal{O}}_{S,s}$ such that everything in sight is defined over $A$, i.e. $\cal X$, the \'etale cover $\{U_i \rightarrow \cal X \}$ and the \'etale maps 
\[ U_i \rightarrow \spec{\frac{A[x_1,\ldots,x_{d+1}]}{(x_1\cdots x_r -t)}} .\]
Now by Artin's approximation theorem, the point $\alpha: \spec{\widehat{\cal{O}}_{S,s}}\rightarrow \spec{A}$ of the finite type $R$-scheme $\spec{A}$, coming from the inclusion $A\subset \widehat{\cal{O}}_{S,s}$, can be approximated modulo $t^n$ by a Henselian point, in other words there exists a morphism $\alpha^{h,n}:\spec{\cal{O}^h_{S,s}}\rightarrow \spec{A}$ which agrees with $\alpha$ modulo $t^n$. Then pulling back via $\alpha^{h,n}$ and using the fact that $t$ is a local parameter at $s$ we can see that there exists a semistable family $\cal{Y}\rightarrow \spec{\cal{O}_{S,s}^h}$ which agrees with $\cal X$ modulo $t^n$. 

Now finally we note that there must be some finitely generated sub-algebra $B\subset \cal{O}_{S,s}^h$, \'etale over $S$ such that $\cal{Y}$ is defined over $B$ and smooth away from $s$, setting $U=\spec{B}$ then completes the proof.
\end{proof}

To apply this result, we will specialise, and let $R\cong k\pow{t}$ be the ring of integers inside $F\cong k\lser{t}$. Suppose that we are given a semistable scheme $\cal{X} /R$. We will let $\cal{X}^\times$ denote $\cal X$ endowed with the log structure given by the special fibre, and $X_0^\times$ the special fibre endowed with the inverse image log structure from $\cal{X}$. Let $X$ denote the generic fibre of $\cal{X}$. Then we may consider, as in \cite{Nak98}, the log-\'etale cohomology $H^i_\et(X^{\times,\mathrm{tame}}_0,\Q_\ell)$ of the log scheme $X^\times_0$, as an $\ell$-adic representation of $G_F$ (on which the wild inertia group $P_F\subset I_F$ acts trivially). We have the following logarithmic analogue of the smooth and proper base change theorem.

\begin{proposition}[\cite{Nak98}, Proposition (4.2)] \label{llc} There is an isomorphism 
\[ H^i_\et(X^{\times,\mathrm{tame}}_0,\Q_\ell) \cong H^i_\et(X_{F^\mathrm{sep}},\Q_\ell) \]
of $G_{F}$-representations.
\end{proposition}

Of course, there is a similar result for $p$-adic cohomology, which goes as follows. Let $\underline{\mathbf{M}\Phi}_K^N$ denote the category of $(\varphi,N)$-modules over $K$, that is finite dimensional vector spaces together with a semilinear, bijective  Frobenius $\varphi$ and a nilpotent endomorphism $N$, such that $N\varphi=q\varphi N$. Given such a $(\varphi,N)$-module $V$, we can put a connection on $V\otimes_K \cal{R}_K$ by setting $\nabla(v\otimes 1)= N(v)\otimes t^{-1}$, and this induces a fully faithful functor
\[ -\otimes_K \rk : \underline{\mathbf{M}\Phi}_K^N\rightarrow \underline{\mathbf{M}\Phi}_{\rk}^\nabla \]
from $(\varphi,N)$-modules over $K$ to $(\varphi,\nabla)$-modules over $\rk$, whose essential image is exactly the objects for which the connection acts unipotently. For $\cal X$ as above, we may consider the log-crystalline cohomology $H^i_{\log\text{-}\mathrm{cris}}(X^\times_0/K^\times):=H^i_{\log\text{-}\mathrm{cris}}(X_0^\times/W^\times)\otimes K$ of its special fibre as a $(\varphi,N)$-module over $K$, as in \cite{HK94} (the notation $W^\times$, $K^\times$ refers to the fact that we are endowing $W$ with the Teichm\"uller lift of the log structure of the punctured point).

\begin{proposition}[\cite{LP16}, Theorem 5.46] \label{plc} There is an isomorphism
\[H^i_{\log\text{-}\mathrm{cris}}(X^\times_0/K^\times) \otimes_K \cal{R} \cong H^i_{\rig}(X/\rk) \]
inside $ \underline{\mathbf{M}\Phi}_{\rk}^\nabla $. 
\end{proposition}

We therefore obtain the following.

\begin{corollary} \label{rg} Let $\cal{X} \rightarrow \spec{R}$ be a proper, semistable scheme with generic fibre $X$. Then there exists a smooth, geometrically connected curve $C/k$, a $k$-rational point $c\in C(k)$, a proper, semistable scheme $\cal{Y}\rightarrow C$, smooth away from $c$, and an isomorphism $R\cong \widehat{\cal{O}}_{C,c}$ (inducing $F\cong \widehat{k(C)}_c$) such that
\[ H^i_\et(X_{F^\mathrm{sep}},\Q_\ell) \cong H^i_\et(\cal{Y}_{F^\mathrm{sep}},\Q_\ell) \]
as $G_{F}$ representations for all $i,\ell\neq p$, and
\[ H^i_\rig(X/\rk) \cong H^i_\rig(\cal{Y}_{F}/\rk)  \]
as $\pn$-modules, for all $i$. In particular, $C_{\ell,w}(X,H^i)$ is equivalent to $C_{\ell,w}(\cal{Y}_F,H^i)$
\end{corollary}

\begin{proof} By Proposition \ref{spread} we know that there exists a globally defined semistable scheme which agrees with $\cal{X}$ up to order $2$, hence the special fibres coincide as log schemes. Hence by Propositions \ref{llc} and \ref{plc} the cohomologies also coincide.
\end{proof}

Our primary interest here is in $\ell$-independence results, however, let us note in passing that Corollary \ref{rg} implies that the Clemens--Schmid exact sequence, obtained by the first author in collaboration with Tsuzuki in \cite{CT14}, exists for $\cal{X}$. 

\begin{corollary} Let $d=\dim(\cal{X}/R)=\dim \cal X-1$. There is a Clemens--Schmid long exact sequence 
\begin{align*} 
 \ldots \rightarrow H^m_\rig(X_0/K)\rightarrow H^m_{\log\text{-}\mathrm{cris}}(X_0^\times/&K^\times) \overset{N}{\rightarrow} H^m_{\log\text{-}\mathrm{cris}}(X_0^\times/K^\times)(-1) \\
  &\rightarrow H^{2d-m}_\rig(X_0/K)^\vee(-d-1)\rightarrow \ldots
\end{align*}  
\end{corollary}

\begin{proof} When $\cal{X}$ extends to a global semistable family, this is the main result of \cite{CT14}. Given Proposition \ref{spread} it therefore suffices to note that all terms in the exact sequence only depend on the log scheme $X_0^\times$.
\end{proof}

To return to our main focus, then, to prove Theorem \ref{main2} it suffices to do so in the `globally defined' case, i.e. we may assume that we have a smooth geometrically connected curve $C/k$, $c\in C(k)$ such that $F=\widehat{k(C)}_c$, and a semistable scheme $\cal{X}\rightarrow C$, smooth away from $c$, such that $X=\cal{X}_F$.

Let $U=C\setminus c$, and consider a collection $\{\cal{F}_\ell \}_\ell$ of local systems on $U$, where $\cal{F}_\ell$ is a lisse $\Q_\ell$ sheaf on $U$ for $\ell\neq p$ and $\cal{F}_p$ is an overconvergent $F$-isocrystal on $U/K$ (here $F$ is the $q=\#k$-power Frobenius). Then for any closed point $x\in U$ the action of geometric Frobenius on $\cal{F}_{\ell,\bar{x}}$ for $\ell\neq p$, plus the action of $F^{\deg (x)}$ on $\cal{F}_{p,x}$, defines a collection $\{\cal{F}_{\ell,x} \}_\ell$ of representations of the Weil group $W_{k(x)}$.

\begin{definition} \begin{enumerate} \item We say that the system $\{ \cal{F}_\ell\}$ is weakly $\Q$-compatible if $\{\cal{F}_{\ell,x} \}_\ell$ is so for all closed points $x\in U$.
\item We say that the system $\{ \cal{F}_\ell\}$ is pure of weight $n$ if for each closed point $x\in U$ the eigenvalues of $\mathrm{Frob}_{k(x)}$ (resp. $F^{\deg(x)}$) on $\cal{F}_{\ell,\bar{x}}$ (resp. $\cal{F}_{p,x}$) are Weil numbers of weight $n$.
\end{enumerate}
\end{definition}

Given such a system, we may also restrict to the Weil--Deligne group at $c$. When $\ell\neq p$, this is fairly standard, since we may consider the Galois group $G_F$ as a decomposition group at $c$, and hence restrict the $\ell$-adic representation $\pi_1(U,\bar{\eta})\rightarrow \cal{F}_{\ell,\bar\eta}$ of the fundamental group of $U$ (based at some geometric generic point) to get an $\ell$-adic representation of $G_F$, and hence a Weil--Deligne representation, which we shall call $\cal{F}_{\ell,c}$. When $\ell=p$ we use the construction of \cite[\S6.1]{Tsu98}. Pulling back to an appropriate tubular neighbourhood of $c$ gives a functor
\[ F\text{-}\mathrm{Isoc}^\dagger(U/K)\rightarrow \underline{\mathbf{M}\Phi}^\nabla_{\cal{R}_{c}}\]
where $\cal{R}_{c}$ is a copy of the Robba ring at $c$. Then applying Marmora's functor from $\pn$-modules over the Robba ring to Weil--Deligne representations \cite{Mar08} gives us a $K^\mathrm{un}$-valued representation of $W_F'$ which we shall call $\cal{F}_{p,c}$. The key result is then the following, whose proof we will defer to \S\ref{dp} below.

\begin{proposition} \label{gllo} If $\{\cal{F}_\ell \}_\ell $ is a weakly $\Q$-compatible system on $U$, which is moreover pure of some weight $n$, then $\{\cal{F}_{\ell,c} \}_\ell $ is a weakly $\Q$-compatible system of Weil--Deligne representations.
\end{proposition}

We can now complete the proof of Theorem \ref{main2}

\begin{proof}[Proof of Theorem \ref{main2}]
We may assume that $X$ is globally defined, i.e. we have $C,c,U,\cal X$ as above such that $F=\widehat{k(C)}_c$ and $\cal{X}_F\cong X$. We will apply Proposition \ref{gllo}, taking $\cal{F}_\ell$ to be the $i$th higher direct image of the constant sheaf under the map $f:\cal{X}|_U\rightarrow U$. Hence for $\ell\neq p$ we take $\cal{F}_\ell= \mathbf{R}^if_*\Q_\ell$, and for $\ell=p$ we take $\cal{F}_p$ to be the overconvergent $F$-isocrystal $\mathbf{R}^if_*\cal{O}_{X/K}^\dagger$ on $U/K$ constructed by Matsuda and Trihan in \cite{MT04}. By smooth and proper base change in \'etale cohomology we have, for any closed point $x\in U$
\[ \cal{F}_{\ell,\overline{x}} = H^i_\et(\cal{X}_{\bar x},\Q_\ell) \]
and by the same result for Ogus' convergent cohomology \cite[Proposition 3.5]{Ogu84} we have
\[ \cal{F}_{p,x} = H^i_\rig(\cal{X}_x/K(x)) \]
where $K(x)$ is the fraction field of the Witt vectors of $k(x)$. Hence by $C_{\ell,w}(\cal{X}_x,H^i)$ we know that $\{ \cal{F}_\ell \}_\ell$ is a $\Q$-compatible system on $U$, and by Proposition \ref{gllo} it follows that $\{ \cal{F}_{\ell,c} \}_\ell$ is a weakly $\Q$-compatible system of Weil--Deligne representations. But again applying smooth and proper base change in \'etale cohomology tells us that $\cal{F}_{\ell,c}\cong H^i_\ell(X)$ for $\ell\neq p$, and the proof of \cite[Proposition 5.52]{LP16} tells us that $\cal{F}_{p,c} = H^i_p(X) $. The result follows.
\end{proof}

\section{Cohomology over equicharacteristic local fields II} \label{clf2}

In this section, we generalise Theorem \ref{main2} to apply to all smooth and proper varieties over an equicharacteristic local field, not just those with semistable reduction.

\begin{theorem}\label{main22} Suppose that $F\cong k\lser{t}$ is equicharacteristic local, and let $X/F$ be a smooth, proper variety. Then $C_{\ell,w}(X,H^i)$ holds. 
\end{theorem}

The proof will be via a rather delicate weight argument for Weil--Deligne representations. We will therefore let $\Q'_\ell$ denote the field $\Q_\ell$ if $\ell\neq p$, and $K^\mathrm{un}$ if $\ell=p$, and throughout consider $\Q_\ell'$-valued Weil--Deligne representations (i.e. $\Q_\ell'$-valued representations of $W_F'$). If
\[(V,N,\rho:W_F\rightarrow \mathrm{GL}(V)) \]
is such a representation then we have a canonical filtration
\[ M_k:=\sum_{k=i-j} \ker\;N^{i+1} \cap\; \mathrm{im} \;N^j\]
on $V$, called the monodromy filtration. This is functorial in $V$ in the sense that if $f:V\rightarrow W$ is a morphism of Weil--Deligne representations, then $f(M_kV)\subset M_kW$ for all $k$. Moreover each $M_k$ is preserved by the $W_F$-action, and hence there is an induced action of $W_F$ on $\mathrm{Gr}_k^MV$. 

\begin{definition} We say that $V$ is quasi-pure of weight $i$ if the eigenvalues of any lift $\varphi\in W_F$ of geometric Frobenius are $q$-Weil numbers of weight $i$. By convention, the zero object is pure of \emph{all} weights.
\end{definition}

Since the inertia $I_F$ acts through a finite quotient, this does not depend on the choice of Frobenius lift $\varphi$. We will need the following results concerning purity and Weil--Deligne representations.

\begin{proposition} \label{qpwd} \begin{enumerate}
\item \label{1} The subcategory $\mathrm{Rep}_{\Q_\ell'}(W_F')^{(i)}\subset \mathrm{Rep}_{\Q_\ell'}(W_F')$ of Weil--Deligne representations which are pure of some fixed weight $i$ is abelian, i.e. stable under taking kernels and cokernels.
\item \label{2} Any morphism $f:V\rightarrow W$ between Weil--Deligne representations which are pure of weights $i>j$ respectively is zero.
\item \label{3} Let $X/F$ be smooth and proper. Then the $\Q_\ell'$-valued Weil--Deligne representation $H^i_\ell(X)$ attached to the $\ell$-adic cohomology of $X$ is quasi-pure of weight $i$.
\end{enumerate}
\end{proposition}

\begin{proof} If $\ell=p$ then (\ref{1}) is \cite[Lemma 5.61]{LP16}, and (\ref{2}) is \cite[Lemma 5.59]{LP16}; exactly the same proof works for $\ell\neq p$. If $\ell\neq p$ then (\ref{3}) follows from \cite[Theorem 1.1]{Ito05}, and if $\ell=p$ from \cite[Theorem 5.33]{LP16}.
\end{proof}

We will also need the following result.

\begin{lemma} \label{ind} Let $F'/F$ be a finite extension (not necessarily Galois or even separable), and let $X'/F'$ be smooth and proper. Let $X$ denote $X'$ considered as a variety over $F$, and fix a prime $\ell$.
\begin{enumerate}
\item \label{i} If the $\ell$-adic cohomology $H^i_\ell(X')$ is quasi-pure of weight $i$ as a $W'_{F'}$-representation, then $H^i_\ell(X)$ is quasi-pure of weight $i$ as a $W'_F$-representation.
\item \label{ii} If $C_{\ell,w}(H^i)$ holds for $X'/F'$, then $C_{\ell,w}(H^i)$ holds for $X/F$. 
\end{enumerate}
\end{lemma}

\begin{proof}
It suffices to treat the cases where $F'/F$ is separable, and where $F'=F^{1/p}$. If $F'/F$ is separable, then we may view $W'_{F'} \subset W_F'$ as a finite index subgroup. Then we have that $H^i_\ell(X)=\mathrm{Ind}_{W'_{F'}}^{W_F'}H^i_\ell(X')$ is simply the induced representation, which easily implies both (\ref{i}) and (\ref{ii}).

Next suppose that $F'=F^{1/p}$, or in other words that we are looking at the finite extension $\mathrm{Frob}:F\rightarrow F$. Let $X_1$ be a given smooth and proper variety over $F$, and $X_0$ the $F$-variety by considering $X_1$ with the composite structure morphism $X_1\rightarrow \spec{F} \overset{\spec{\mathrm{Frob}}}{\rightarrow} \spec{F}$. Since the map $X_1\hookrightarrow X_0\otimes_{F,\mathrm{Frob}} F$ is a nilpotent thickening, we have $H^i_\ell(X_1)\cong H^i_\ell(X_0\otimes_{F,\mathrm{Frob}} F)$ as $W'_F$-representations, and since the Frobenius map $X_0\rightarrow X_0\otimes_{F,\mathrm{Frob}} F$ is a universal homeomorphism, it follows that $H^i_\ell(X_0)\cong H^i_\ell(X_0\otimes_{F,\mathrm{Frob}} F)$ as $W'_F$-representations. Hence we can deduce that $H^i_\ell(X_1)\cong H^i_\ell(X_0)$ as $W'_F$-representations, which allows us to conclude.
\end{proof}

With these preliminaries in place, we can now prove Theorem \ref{main22}.

\begin{proof}[Proof of Theorem \ref{main22}]
let $X/F$ be smooth and proper. Then by de Jong's theorem on alterations \cite[Theorem 8.2]{dJ96}, we may choose a proper hypercover $X_\bu\rightarrow X$ and finite extensions $F_n/F$ such that:
\begin{itemize}
\item the structure morphism $X_n\rightarrow \spec{F}$ factors through $\spec{F_n}$;
\item $X_n/F_n$ has strictly semistable reduction.
\end{itemize}
In particular, by applying Theorem \ref{main2}, Proposition \ref{qpwd}(\ref{3}) and Lemma \ref{ind} we can see that the cohomology groups $H^i_\ell(X_n)$, where $X_n$ is viewed as an $F$-variety, are quasi-pure of weight $i$, and moreover that they satisfy weak $\ell$-independence. We now consider the spectral sequence
\[ E_1^{n,i} = H_\ell^i(X_n)\Rightarrow H_\ell^{i+n}(X).  \]
Since $E_1^{n,i}$ is quasi-pure of weight $i$, it follows from Proposition \ref{qpwd}(\ref{1}) that $E_2^{i,n}$ is also quasi-pure of weight $n$, and hence from Proposition \ref{qpwd}(\ref{2}) that the spectral sequence degenerates at $E_2$.

We will now consider the induced filtration $P_\bu$ on $H^{i+n}_\ell(X)$, whose associated graded
\[ \mathrm{Gr}_i^PH^{i+n}_\ell(X)=E_2^{n,i} \]
is therefore quasi-pure of weight $i$. If we look at the quotient map
\[ H^{i+n}_\ell(X) \rightarrow \mathrm{Gr}_{i+n}^P = E_2^{0,i+n} \]
then it follows from Proposition \ref{qpwd}(\ref{1}) that the kernel $P_{i+n-1}H^{i+n}_\ell(X)$ is quasi-pure of weight $i+n$, and hence that the quotient map
\[ P_{i+n-1}H^{i+n}_\ell(X) \twoheadrightarrow \mathrm{Gr}_{i+n-1}^P  = E_2^{1,i+n-1} \]
must be zero by Proposition \ref{qpwd}(\ref{2}). In particular, we have $P_{i+n-1}H^{i+n}_\ell(X)=P_{i+n-2}H^{i+n}_\ell(X)$, and again applying Proposition \ref{qpwd}(\ref{2}) we can see that the quotient
\[ P_{i+n-2}H^{i+n}_\ell(X) \twoheadrightarrow \mathrm{Gr}_{i+n-2}^P  = E_2^{2,i+n-2} \]
must be zero. Continuing thus, we find that $P_{i+n-j}H^{i+n}_\ell(X)=P_0H^{i+n}_\ell(X)=0$ for all $j>0$, from which we deduce that the $E_2$-page of the spectral sequence only has non-zero terms in the first column. Put differently, we have a long exact sequence
\[ 0\rightarrow H^{i+n}_\ell(X)\rightarrow H^{i+n}_\ell(X_0)\rightarrow H^{i+n}_\ell(X_1)\rightarrow H^{i+n}_\ell(X_2)\rightarrow \ldots \]
for all $i,n,\ell$. Since we know weak $\ell$-independence for all the terms $H^{i+n}_\ell(X_n)$, we can therefore deduce weak $\ell$-independence for $H^{i+n}_\ell(X)$.
\end{proof}

\section{\texorpdfstring{$L$}{L}-functions and the proof of Proposition \ref{gllo}}\label{dp}

The purpose of this section is to prove Proposition \ref{gllo}, as well as give an application of Theorem \ref{main22} showing that the Hasse--Weil $L$-functions of smooth and proper varieties over global function fields (including those defined using $p$-adic cohomology) are independent of $\ell$. The basic idea is that away from $\ell=p$ Proposition \ref{gllo} follows from \cite[Th\'eor\`eme 9.8]{Del73}, and in fact Deligne's proof also works at $\ell=p$ once certain facts about $p$-adic cohomology of smooth curves are in place. The deduction of $\ell$-independence for the Hasse--Weil $L$-function is then entirely straightforward. In fact, it is not entirely accurate to describe this as an application of Theorem \ref{main22}, since it is in fact a version of $\ell$-independence for local $L$-functions that forms a key component of the \emph{proof} of Proposition \ref{gllo}, and therefore of Theorem \ref{main22}. However, we thought it worth isolating this special case. All the results presented in this section are more or less special cases of more general results of Abe--Caro proved in \cite{AC14}, however, in the simple case that we are interested in one can give proofs using simpler machinery from the theory of arithmetic $\cal{D}$-modules on curves and formal curves, in particular that from \cite{Car06b} and \cite{Cre12}.  

Since the situation we are interested in in this section is somewhat different to that in the rest of the article, we will change notation slightly (only for this section, we will return to the usual notations in the next section). So we will let $F$ be a global function field with field of constants $k$, and let $X$ be a smooth, projective model for $F$. We will let $U$ be an open subset of $X$, and let $S=X\setminus U$. For any place $v\in \norm{X}$ (with some fixed local parameter $t_v$) we will let $k_v$ denote the residue field at $v$, $K_v$ the unique unramified extension of $K$ with residue field $k_v$, $W_v$ its ring of integers and $\cal{R}_v$ a copy of the Robba ring at $v$, considered as a subring of $K_v\pow{t_v,t_v^{-1}}$. Let $\varphi_v$ denote a $\#k_v$-power Frobenius on $\cal{R}_v$. We will also denote $\#k_v$-power Frobenius operators by $F_v$.

If we are given an overconvergent $F$-isocrystal $\cal{F}\in F\text{-}\isoc{U/K}$ and a point $v\in X$ we can consider Tsuzuki's functor
\[ \mathbf{i}_v^*: F\text{-}\isoc{U/K} \rightarrow \underline{\mathbf{M}\Phi}_{\cal{R}_v}^\nabla \]
to the category of $(\varphi_v,\nabla)$-modules over $\cal{R}_v$, and therefore define the local $L$-factor of $\cal{F}$ at $v$ by
\[ L_v(\cal{F},t):=\det(1-t^{\deg v} \varphi_v\mid (\mathbf{i}_v^*\cal{F})^{\nabla=0}).\]
The global $L$-function of $\cal{F}$ is then defined to be
\[ \widehat{L}(\cal{F},t):= \prod_{v\in \norm{X}} L_v(\cal{F},t), \]
the notation $\widehat{L}$ used here is in order to distinguish it from the Etesse--Le Stum $L$-function $L_{\mathrm{EL}}$ defined in \cite{ELS93}, which is defined by only considering local $L$-factors at points of $U$. Note that at places $v\in U$ the local $L$-factor can also be described (by Dwork's trick) as the inverse characteristic polynomial $\det(1-t^{\deg v}F^{\deg v} \mid \cal{F}_v)$ of the linearised Frobenius of $\cal{F}$ acting on the stalk $\cal{F}_v$ of $\cal{F}$ at $v$. Our first task will be to prove a function equation for $\widehat{L}(\cal{F},t)$, and to do so we will need to interpret the local factors $L_v(\cal{F},t)$ in terms of arithmetic $\cal{D}$-modules on $X$.

So we will denote by $\cal{X}$ a lift of $X$ to a smooth projective formal curve over $W$, by $\cal{D}^\dagger_{\cal{X},\Q}$ Berthelot's ring of arithmetic differential operators on $\cal{X}$, and by $F\text{-}\mathrm{Mod}_{\mathrm{hol}}(\cal{D}^\dagger_{\cal{X},\Q})$ (resp. $F\text{-}D^b_\mathrm{hol}(\cal{D}^\dagger_{\cal{X},\Q})$) the category of holonomic $F\text{-}\cal{D}^\dagger_{\cal{X},\Q}$-modules (resp. the bounded derived category of $\cal{D}^\dagger_{\cal{X},\Q}$-modules equipped with Frobenius structure, whose cohomology sheaves are holonomic). For any complex $\cal{K}\in F\text{-}D^b_\mathrm{hol}(\cal{D}^\dagger_{\cal{X},\Q})$ we will denote by $L(\cal{K},t)$ the $L$-function of $\cal{K}$ as defined in the introduction to \cite{Car06b}; let us quickly recall his definition. For $v\in \norm{X}$ we may lift $v$ to a morphism $i_v:\spf{W_v}\rightarrow \cal{X}$, and for any $\cal{K}\in F\text{-}D^b_\mathrm{hol}(\cal{D}^\dagger_{\cal{X},\Q})$ we can consider the pullback $i_v^+\cal{K}\in F\text{-}D^b_\mathrm{hol}(\cal{D}^\dagger_{\spf{W_v},\Q})$. Its cohomology sheaves are therefore $F$-isocrystals over $K_v$, we may therefore consider $F^{\deg v}$ as a $K$-linear operator on these cohomology sheaves. Caro then defines
\begin{align*} L_v(\cal{K},t)&:= \prod_{i\in \Z} \textstyle{\det_K}(1-t^{\deg v}F^{\deg v}|H^i(i_v^+\cal{K}) )^{\frac{(-1)^{r}}{\deg v}} \\
L(\cal{K},t)&:=\prod_{v\in \norm{X}} L_v(\cal{K},t).\end{align*} 
The functional equation that we require will follow from the existence of the `intermediate extension' functor
\[ j_{!+}:F\text{-}\mathrm{Isoc}(U/K)\rightarrow F\text{-}\mathrm{Mod}_{\mathrm{hol}}(\cal{D}^\dagger_{\cal{X},\Q}) \]
which will be such that:
\begin{enumerate} \item $\widehat{L}(\cal{F},t)=L(j_{!+}\cal{F},t)$;
\item $\mathbf{D}(j_{!+}\cal{F})\cong j_{!+}(\cal{F}^\vee)$, where $\mathbf{D}$ is the dual functor for $\cal{D}$-modules, and $(-)^\vee$ that for isocrystals. 
\end{enumerate}
As noted at the beginning of this section, this is a special case of much more general results of \cite{AC14}, however, we thought it worthwhile to go through the construction in detail in the simple case that we are interested in. To construct $j_{!+}$, first let us consider Berthelot's ring $\cal{D}^\dagger_{\cal{X},\Q}(^\dagger S)$ of arithmetic differential operators on $\cal{X}$ with overconvergent singularities along $S$, and the corresponding categories $F\text{-}\mathrm{Mod}_\mathrm{hol}(\cal{D}^\dagger_{\cal{X},\Q}(^\dagger S))$ and $F\text{-}D^b_\mathrm{hol}(\cal{D}^\dagger_{\cal{X},\Q}(^\dagger S))$ of (complexes of) holonomic $F\text{-}\cal{D}^\dagger_{\cal{X},\Q}(^\dagger S)$-modules. Then thanks to \cite[Th\'eor\`eme 2.3.3]{Car06b} we may consider any $\cal{F}\in F\text{-}\isoc{U/K}$ as a holonomic $F\text{-}\cal{D}^\dagger_{\cal{X},\Q}(^\dagger S)$-module, which we shall do. We have canonical functors
\begin{align*}  j_+: F\text{-}D^b_\mathrm{hol}(\cal{D}^\dagger_{\cal{X},\Q}(^\dagger S)) &\rightarrow F\text{-}D^b_\mathrm{hol}(\cal{D}^\dagger_{\cal{X},\Q}) \\
 j^+: F\text{-}D^b_\mathrm{hol}(\cal{D}^\dagger_{\cal{X},\Q}) &\rightarrow F\text{-}D^b_\mathrm{hol}(\cal{D}^\dagger_{\cal{X},\Q}(^\dagger S))
 \end{align*}
given by restriction and extension of scalars along $\cal{D}^\dagger_{\cal{X},\Q}\rightarrow \cal{D}^\dagger_{\cal{X},\Q}(^\dagger S)$ respectively, these are both exact for the natural $t$-structures on both sides. Define $j_!$ to be $\mathbf{D}j_+\mathbf{D}$ where $\mathbf{D}$ is the duality functor, note that this is exact for holonomic $F\text{-}\cal{D}$-modules by (the proof of) \cite[Proposition III.4.4]{Vir00}. By \cite[Proposition I.4.4]{Vir00} we have $j^+\mathbf{D}=\mathbf{D}j^+$ and hence we get adjoint pairs $(j^+,j_+)$ and $(j_!,j^+)$. Also by \cite[1.1.8]{Car06b} we have $j^+j_!\cong j^+j_+\cong \mathrm{id}$. 

The closed immersion $i:S\rightarrow X$ lifts to a closed immersion $i:\cal{S}\rightarrow \cal X$ and hence we obtain Berthelot's pullback and pushforward functors 
\begin{align*} i^!:F\text{-}D^b_\mathrm{hol}(\cal{D}^\dagger_{\cal{X},\Q}) &\rightarrow F\text{-}D^b_\mathrm{hol}(\cal{D}^\dagger_{\cal{S},\Q}) \\
i_+:F\text{-}D^b_\mathrm{hol}(\cal{D}^\dagger_{\cal{S},\Q}) &\rightarrow F\text{-}D^b_\mathrm{hol}(\cal{D}^\dagger_{\cal{X},\Q}),
\end{align*}
we define $i^+=\mathbf{D}i^!\mathbf{D}$. It follows from \cite[Th\'eor\`eme 4.3.13]{Ber02} that $i_+\mathbf{D}=\mathbf{D}i_+$ and hence by \cite[Th\'eor\`eme 1.2.21]{Car06b} we have two adjoint pairs $(i^+,i_+)$ and $(i_+,i^!)$. Also by \cite[Th\'eor\`eme 5.3.3]{Ber02} we get $i^!i_+\cong i^+i_+\cong \mathrm{id}$. By \cite[(1.1.6.5)]{Car06b} we have, for any $\cal{E}\in F\text{-}D^b_\mathrm{hol}(\cal{D}^\dagger_{\cal{X},\Q})$, exact triangles 
\begin{align*} j_!j^+\cal{E} \rightarrow &\cal{E} \rightarrow i_+i^+\cal{E}\overset{+1}{\rightarrow }  \\
i_+i^!\cal{E} \rightarrow &\cal{E} \rightarrow j_+j^+\cal{E}\overset{+1}{\rightarrow }. \end{align*}
In particular taking $\cal E=j_+ \cal F$ we get an exact triangle $j_!\cal{F} \rightarrow j_+\cal{F} \rightarrow i_+i^+j_+\cal{F}\overset{+1}{\rightarrow } $, and we define $j_{!+}\cal{F}:=\mathrm{im}(j_!\cal{F} \rightarrow j_+\cal{F})$, note that this makes sense as a $\cal{D}_{\cal{X},\Q}^\dagger$-module because $j_+$ and $j_!$ are exact. By exactness of $\mathbf{D}$ for holonomic $F\text{-}\cal{D}^\dagger$-modules we have $\mathbf{D}j_{!+}\cal{F}\cong j_{!+}\mathbf{D}\cal{F}$, and hence using the main result of \cite{Car05} we deduce that $\mathbf{D}(j_{!+}\cal{F})\cong j_{!+}(\cal{F}^\vee)$. To compare the $L$-functions of $\cal{F}$ and $j_{!+}\cal F$ essentially boils down to the following local calculation from \cite{Cre12}.

\begin{proposition} \label{locallf} Let $v\in X$, and let $i_v: \spf{W_v}\rightarrow \cal{X}$ be a lifting of the inclusion $v\rightarrow X$. 
\begin{enumerate} \item If $v\in U$ then $(i_v^+j_{!+}\cal{F})[-1]\cong \cal{F}_v$ as $F_v$-isocrystals over $K_v$. 
\item If $v\in S$ then $(i_v^+j_{!+}\cal{F})[-1]\cong (\mathbf{i}_v^*\cal{F})^{\nabla=0}$ as $F_v$-isocrystals over $K_v$.
\end{enumerate}
\end{proposition}

\begin{remark}\begin{enumerate}\item In both cases,  $i_v^+j_{!+}\cal{F}$ is naturally an $F$-isocrystal, taking a suitable power of the Frobenius allows us to regard it as a $F_v$-isocrystal. In the second case, $\mathbf{i}_v^*\cal{F}$ is the $(\varphi_v,\nabla)$-module over $\cal{R}_v$ associated to $\cal{F}$ as above. Note also that implicit in both claims is that $i_v^+j_{!+}\cal{F}$ is concentrated in degree $-1$.
\item The proposition implies that the local $L$-factors of $\cal F$ and $j_{!+}\cal F$ agree at all places of $X$, hence that $\widehat{L}(\cal F,t)=L(j_{!+}\cal F,t)$.
\item By Dwork's trick, the formula in the second case is in fact also valid in the first.
\end{enumerate}
\end{remark}

\begin{proof} The first case is noted in the proof of \cite[Proposition 3.3.1]{Car06b} (in combination with the main result from \cite{Car05}), we will consider the second. The key point is to show that we may calculate $j_!$ and $j_+$ entirely locally. So let $\mathbf{i}_v: \mathfrak{X}_v:=\spf{W_v\pow{t_v}}\rightarrow \cal X$ be a lifting of the map $\spf{k_v\pow{t_v}}\rightarrow X$ of the completed local ring at $v$ into $X$ (previously, $\mathbf{i}_v^*$ was simply notation, however shortly we shall see how to interpret it as a genuine pullback via $\mathbf{i}_v$). Let $\cal{D}^\dagger_v$ denote the ring of arithmetic differential operators on $\mathfrak{X}_v$ as constructed by Crew in \cite{Cre12}, $\cal{D}^\dagger_v(v)$ the version with overconvergent singularities along the closed point of $\mathfrak{X}_v$. Then (unlike the usual case in $\cal{D}$-module theory), the map $\mathbf{i}_v$ is actually a map of ringed spaces $(\mathfrak{X}_v,\cal{D}_v^\dagger)\rightarrow (\cal X,\cal{D}_{\cal X,\Q}^\dagger)$ (resp. $(\mathfrak{X}_v,\cal{D}_v^\dagger(v))\rightarrow (\cal X,\cal{D}_{\cal X,\Q}^\dagger(^\dagger S))$) and hence we have a natural pullback functors
\begin{align*}  \cal E &\mapsto \cal{E} \otimes_{\cal{D}_{\cal X,\Q}^\dagger} \cal{D}_v^\dagger =:\mathbf{i}_v^*\cal E \\
\cal E &\mapsto \cal{E} \otimes_{\cal{D}_{\cal X,\Q}^\dagger(^\dagger S)} \cal{D}_v^\dagger(v) =:\mathbf{i}_v^*\cal E \end{align*}
from holonomic $F\text{-}\cal{D}_{\cal X,\Q}^\dagger$-modules to holonomic $F\text{-}\cal{D}_v^\dagger$-modules (resp. $F\text{-}\cal{D}_{\cal X,\Q}^\dagger(^\dagger S)$-modules to $F\text{-}\cal{D}_v^\dagger(v)$-modules). Note that we now have $\mathbf{i}_v^*$ denoting three distinct functors, however, we shall shortly see that they are all compatible, so no essential confusion should arise.

As in the global case, we may define functors $j_{v+}$ and $j_v^+$ to be restriction and extension of scalars respectively along the natural map $\cal{D}^\dagger_v\rightarrow \cal{D}^\dagger_v(v)$, as well as a duality functor $\mathbf{D}$ using the machinery of \cite{Vir00}. Defining $j_{v!}=\mathbf{D}j_{v+}\mathbf{D}$ we then get adjoint pairs $(j_v^+,j_{v+})$ and $(j_{v!},j_v^+)$ as before, and hence we can define $j_{v!+}\cal{E}=\mathrm{im}(j_{v!}\cal{E}\rightarrow j_{v+}\cal{E})$ for holonomic $F\text{-}\cal{D}^\dagger_v(v)$-module $\cal E$. Applying \cite[Proposition I.4.4]{Vir00} to the map of rings $\mathbf{i}_v^{-1}\cal{D}_{\cal X,\Q}^\dagger \rightarrow \cal{D}_v^\dagger$ (resp. $\mathbf{i}_v^{-1}\cal{D}_{\cal X,\Q}^\dagger(^\dagger S) \rightarrow \cal{D}_v^\dagger(v)$) we get $\mathbf{D}\mathbf{i}_v^*\cong \mathbf{i}_v^*\mathbf{D}$. 

\begin{claimu} For any coherent $\cal{D}_{\cal{X},\Q}^\dagger(^\dagger S)$-module $\cal E$ we have
\[ \mathbf{i}_v^*j_+\cal E\cong j_{v+}\mathbf{i}_v^*\cal E.\]
\end{claimu}

\begin{proof}[Proof of Claim] In the usual manner we may reduce to the case of coherent $\cal{D}_{X_n}^{(m)}(S)$-modules $\cal E$, where $X_n=\cal X\otimes_{W} W/p^{n+1}$ and
\[ \cal{D}_{X_n}^{(m)}(S):= \cal B(t_S,p^{m+1} )\otimes \cal{D}_{X_n}^{(m)}\]
is the ring of differential operators constructed in \cite[\S4.2.5]{Ber96a} (here $t_S$ is a section of $\cal{O}_{\cal X}$ cutting out $S$). Defining the local version
\[ \cal{D}_{v,n}^{(m)}(v) := \cal B(t_v,p^{m+1} )\otimes \cal{D}_{v,n}^{(m)} \]
entirely similarly, one easily verifies the isomorphism
\[ \cal{D}_{v,n}^{(m)}(v)\cong \cal{D}_{v,n}^{(m)}\otimes_{\mathbf{i}_v^{-1}\cal{D}_{X_n}^{(m)}} \mathbf{i}_v^{-1}\cal{D}_{X_n}^{(m)}(S) \]
and the claim follows.
\end{proof}

Combining this with flatness of $\mathbf{i}_v^{-1}\cal{D}_{\cal X,\Q}^\dagger\rightarrow \cal{D}^\dagger_v$ we can see that we get an isomorphism
\[ j_{v!+}(\mathbf{i}_v^*\cal{E}) \cong \mathbf{i}_v^*(j_{!+}\cal{E})\]
for any holonomic $F\text{-}\cal{D}^\dagger_{\cal{X},\Q}(^\dagger S)$-module $\cal E$. Let $i_v:\spf{W_v} \rightarrow \cal X$ denote the composition of $\mathbf{i}_v$ with the inclusion of the zero section $\spf{W_v}\rightarrow \mathfrak{X}_v$, so that $i_v^+$ factors through $\mathbf{i}_v^*$. If we let $\cal{E}^\dagger_v\subset \cal{R}_v$ denote a copy of the bounded Robba ring at $v$, then the pullback functor
\[ \mathbf{i}_v^*: F\text{-}\mathrm{Isoc}^\dagger(U/K) \rightarrow \underline{\mathbf{M}\Phi}_{\cal{R}_v}^\nabla \]
factors through the category $\underline{\mathbf{M}\Phi}_{\cal{E}_v^\dagger}^\nabla$ of $(\varphi_v,\nabla)$-modules over $\cal{E}_v^\dagger$, and the diagram
\[ \xymatrix{ F\text{-}\mathrm{Isoc}^\dagger(U/K) \ar[r]^-{\mathbf{i}_v^*}\ar[d] &  \underline{\mathbf{M}\Phi}_{\cal{E}_v^\dagger}^\nabla\ar[d] \\  F\text{-}\mathrm{Mod}_{\mathrm{hol}}(\cal{D}^\dagger_{\cal{X},\Q}(^\dagger S)) \ar[r]^-{\mathbf{i}_v^*} &  F_v\text{-}\mathrm{Mod}_{\mathrm{hol}}(\cal{D}^\dagger_{v}(v))  } \]
commutes, where the right hand vertical arrow is that viewing a $(\varphi_v,\nabla)$-module as a holonomic $F_v\text{-}\cal{D}^\dagger_{v}(v)$-module. The upshot of all these complicated trivialities is that to prove (2), we may reduce to the local case; in other words we may replace $\cal X$ by $\mathfrak{X}_v$.

We may therefore suppose that we are given a $(\varphi_v,\nabla)$-module $M$ over $\cal{E}_v^\dagger$, and we must show that
\[ ( i_v^+j_{v!+} M )[-1] \cong (M \otimes \cal{R}_v)^{\nabla=0} \]
as $F_v$-isocrystals over $K_v$. Note that by \cite[Proposition 5.1.3]{Cre12} the complex $i_v^+j_+M$ is concentrated in degrees 0 and -1. From the exact triangle $j_{v!}M \rightarrow j_{v+}M\rightarrow i_{v+}i_v^{+}j_{v+}M\overset{+1}{\rightarrow}$ we therefore deduce the exact sequence
\[ 0 \rightarrow j_{v!+}M \rightarrow j_{v+}M\rightarrow i_{v+}H^0(i_v^+j_{v+}M)\rightarrow 0 \]
of $F\text{-}\cal{D}^\dagger_v$-modules. Now applying $i_v^+$ and looking at the long exact sequence in cohomology we can see that $i_v^+j_{v!+}M$ is concentrated in degree $-1$ and that $i_v^+j_{!+}M[-1]\cong H^{-1}(i_v^+j_{v+}M)$. Now the claim follows from again applying \cite[Proposition 5.1.3]{Cre12}.
\end{proof}

Hence applying Poincar\'e duality for $F\text{-}\cal{D}^\dagger_{\cal X,\Q}$-modules \cite[Theorem 1.3.5]{Car06b} together with the cohomological formula for $L$-functions \cite[Th\'eor\`eme 3.3.4]{Car06b} gives the following.

\begin{corollary}  For any $\cal{F}\in F\text{-}\isoc{U/K}$ we have
\[ \widehat{L}(\cal{F},t)= \epsilon(\cal{F},t)\widehat{L}(\cal{F}^\vee(1),t^{-1})\]
where $\epsilon(\cal{F},t)$ is a monomial in $t$.
\end{corollary}

The next key ingredient in the proof of Proposition \ref{gllo} is the ability to twist overconvergent $F$-isocrystals by characters of the id\`ele class group. So let $\chi:\A_F^*/F^*\rightarrow \C^*$ be a finite order character, unramified at all places of $U$. Via the isomorphism $\A_F^*/F^* \cong W_F^\mathrm{ab}$ of global class field theory, we may view $\chi$ as a character of the Weil group $W_F$ of $F$, and since the image of $\chi$ consists of roots of unity we may choose a finite extension $L/K$ and view  $\chi$ as a continuous character $\chi: G_F\rightarrow L^*$, unramified at all places of $U$. Hence we obtain a character $\chi:\pi_1^\et(U,\bar\eta)\rightarrow L^*$ where $\bar\eta$ is some choice of geometric generic point.

Now, by \cite[Theorem 2.1]{Cre87}  there is an equivalence of categories between continuous $L$-valued representations of $\pi_1^\et(U,\bar\eta)$ which have finite local monodromy at all places of $S$ and the extension of scalars $F\text{-}\mathrm{Isoc^\dagger}(U/K)\otimes_K L$ of the $K$-linear category $F\text{-}\mathrm{Isoc}^\dagger(U/K)$ to $L$. Hence we may associate to $\chi$ a rank one isocrystal $\cal{E}_\chi \in F\text{-}\mathrm{Isoc}^\dagger(U/K)\otimes_K L$. While we defined $L$-functions for objects in $F\text{-}\mathrm{Isoc}^\dagger(U/K)$, essentially the same definition works for $F\text{-}\mathrm{Isoc}^\dagger(U/K)\otimes_K L$, and hence it makes sense to speak of the twisted $L$-function $L(\cal{F}\chi,t):=L(\cal{F}\otimes \cal{E}_\chi,t)$.

To prove Proposition \ref{gllo} we now proceed \emph{exactly} as in the proof of \cite[Th\'eor\`eme 9.8]{Del73}, which shows that the semisimplifications $\{ \cal{F}_{\ell,c}^\mathrm{ss} \}$ (as Weil--Deligne representations, i.e. not just their Frobenius semisimplifications) are weakly $\Q$-compatible. Hence the traces of $W_F$ of the graded pieces of the \emph{weight} filtration have values in $\Q$ independence of $\ell$, and we can then conclude using the weight monodromy theorem (\cite[Th\'eor\`eme 1.8.4]{Del80} for $\ell\neq p$, \cite[Theorem 10.8]{Cre98} for $\ell=p$), since this states that the weight filtration coincides with the monodromy filtration.

 Before moving on to unipotent fundamental groups over local fields, let us give an `application' of Theorem \ref{main22} to give a new proof of a result of Abe and Caro \cite[Corollary 4.3.13]{AC14} concerning Hasse--Weil $L$-functions of varieties over global function fields (at least in the semistable case). Of course, as remarked earlier this `application' is more of a special case of the \emph{proof} of Proposition \ref{gllo}, however, we thought it was worth isolating as a separate result in its own right. So suppose that we have a smooth and proper variety $Y/F$, then for any place $v$ of $F$, we can base change $Y$ to the completion $F_v$, and consider the rigid cohomology groups $H^i_\rig(Y_{F_v}/\cal{R}_v)$ as $(\varphi_v,\nabla)$-modules over $\cal{R}_v$.

\begin{definition} Define
\begin{align*}  L_{\mathrm{HW},v}(F,H^i_\rig(Y),t)&:=\det(1-\varphi_v t^{\deg v} \mid H^i_\rig(Y_{F_v}/\cal{R}_v)^{\nabla=0}) \\
L_{\mathrm{HW}}(F,H^i_\rig(Y),t) &= \prod_v L_{\mathrm{HW},v}(F,H^i_\rig(Y),t). \end{align*} 
\end{definition}

There is another $p$-adic definition of the Hasse--Weil $L$-function given by Abe and Caro in \cite[\S4.3]{AC14}, which they denote $L_{\mathrm{HW}}(F,\cal{H}^i(Y/K),t)$. To see that these two definitions coincide, given Proposition \ref{locallf} above, it suffices to note that if we have a smooth projective model $\cal{Y}\rightarrow U$ over \emph{some}  
open subset of $U$, and let $\cal{H}^i$ denote Matsuda and Trihan's overconvergent pushforward $\mathbf{R}^if_*\cal{O}_{Y/K}^\dagger$, then for any place $v$ of $F$, we have $\mathbf{i}_v^*\cal{H}^i\cong H^i_\rig(Y_{F_v}/\cal{R}_v)$. This was noted in the proof of \cite[Proposition 5.52]{LP16}. We have the following `corollary' of Theorem \ref{main22}.

\begin{corollary} \label{hwi} Let $Y/F$ be a smooth and proper variety. Then $L_{\mathrm{HW}}(F,H^i_\rig(Y),t)$ agrees with the Hasse--Weil $L$-function $L_{\mathrm{HW}}(F,H^i_\et(Y_{F^\mathrm{sep}},\Q_\ell),t)$ attached to the $\ell$-adic cohomology of $Y$.
\end{corollary}

\section{Unipotent fundamental groups over equicharacteristic local fields}\label{uel}

We will now suppose that we have $F\cong k\lser{t}$ equicharacteristic local, the main result of this section is the following.

\begin{theorem} \label{main3} Let $X/F$ be smooth and projective, with semistable reduction, and let $x\in X(F)$. Then $C_{\ell,w}(X,\hat{\cal U}/\mathfrak{a}^\bu)$ holds.
\end{theorem}

Of course, the basic strategy of proof will be identical to that in \S\ref{clf}, the first step will be in showing an appropriate analogue of Propositions \ref{llc} and \ref{plc} holds for unipotent fundamental groups. So let $\cal X$, $\cal X^\times$, $X_0^\times$ be as in the paragraph before Proposition \ref{llc}, let $x_0\in X_0^\mathrm{sm}(k)$ denote the specialisation of our given point $x\in X(F)$, and choose a geometric point $\bar{x}_0$ above it. For $\ell\neq p$ we will consider $\pi_1^{\log\text{-}\et}(X_0^{\times,\mathrm{tame}},\bar{x}_0)$, the log-\'etale fundamental group of $X_0^{\times,\mathrm{tame}}$ as defined in \cite[\S2]{Lep13}. This has a natural $G_F$-action (trivial on wild inertia $P_F$), and hence so does its $\Q_\ell$-pro-unipotent completion $\pi_1^{\log\text{-}\et}(X_0^{\times,\mathrm{tame}},\bar{x}_0)_{\Q_\ell}$. 

\begin{proposition}[\cite{Lep13}, Theorems 2.5, 2.7] \label{fgssl} There is an isomorphism
\[ \pi_1^{\log\text{-}\et}(X_0^{\times,\mathrm{tame}},\bar{x}_0)_{\Q_\ell}\cong \pi_1^{\et}(X_{F^\mathrm{sep}},\bar x)_{\Q_\ell}\]
of pro-unipotent groups with $G_F$-action.
\end{proposition}

For $\ell=p$ we will consider the category $\mathrm{Uni}_K(X^\times_0/K^\times)$ of unipotent log-isocrystals on $X^\times_0$, relative to $W$ equipped with the Teichm\"uller lift of the log structure on $k$ defined by $\N\rightarrow k$, $1\mapsto 0$, as defined in \cite[\S4]{Shi00}. This is a Tannakian category by Proposition 4.14. of \emph{loc. cit.}, and $x_0$ defines a fibre functor
\[ x_0^*: \mathrm{Uni}_K(X^\times_0/K^\times)\rightarrow \mathrm{Vec}_K. \]
We let $\pi_1^{\log\text{-}\mathrm{cris}}(X_0^\times/K^\times,x_0)$ denote the corresponding pro-unipotent group. To put a Frobenius and monodromy operator on this pro-unipotent group, i.e. to turn this into a non-abelian $(\varphi,N)$-module, we follow the approach of \cite{Laz16b}. So let $k^\times$ denote $k$ considered with the log structure of the punctured point, and $f_0:X_0^\times \rightarrow \spec{k^\times}$ the structure morphism.

Consider the Tannakian category
\[ \cal{N}F\text{-}\mathrm{Isoc}(X_0^\times/K)\]
of relatively unipotent log-$F$-isocrystals on $X_0$, relative to $W$ equipped with the trivial log structure. In other words
\[ \cal{N}F\text{-}\mathrm{Isoc}(X_0^\times/K)\subset F\text{-}\mathrm{Crys}(X_0^\times/W)_{\Q} \]
is the full sub-category of objects which are iterated extensions of those pulled back from the category $\underline{\mathbf{M}\Phi}_K^N$ of $(\varphi,N)$-modules  via
\[ f_0^*: \underline{\mathbf{M}\Phi}_K^N = F\text{-}\mathrm{Isoc}(\spec{k^\times}/K)\rightarrow F\text{-}\mathrm{Isoc}^\dagger(X_0^\times/K). \]
We therefore have a pair of functors
\[ x_0^*: \cal{N}F\text{-}\mathrm{Isoc}(X_0^\times/K) \leftrightarrows \underline{\mathbf{M}\Phi}_K^N :f^*_0, \]
and using the formalism of \cite[\S2.1]{Laz15}, we may attach to this situation a `relative fundamental group'
\[ G^{\lcris}_{X_0^\times,x_0}:= G(\cal{N}F\text{-}\mathrm{Isoc}(X_0^\times/K),x_0^*)\]
in the terminology of \emph{loc. cit.}; this is a pro-unipotent group scheme over the category $\underline{\mathbf{M}\Phi}_K^N$. Concretely, then such an object is something of the form $\spec{A}$, where $A$ is a Hopf $K$-algebra equipped with the structure of and ind-$(\varphi,N)$-module. Moreover the Hopf-algebra and $(\varphi,N)$-module structures are compatible in the sense that the maps
\begin{align*}
A\otimes_K A \overset{m}{\rightarrow} A,\;\; A\overset{c}{\rightarrow} A\otimes_K A  \\
K \overset{u}{\rightarrow} A,\;\;A\overset{\varepsilon}{\rightarrow} K,\;\; A\overset{\iota}{\rightarrow} A
\end{align*}
expressing the Hopf algebra structure of $A$ are in fact morphisms of $(\varphi,N)$-modules. By \cite[Proposition 2.3]{Laz15} this group scheme $G^{\lcris}_{X_0^\times,x_0}$ is uniquely characterised by the fact that its category of enriched representations, i.e. representations in objects of the category $\underline{\mathbf{M}\Phi}_K^N$, is equivalent to $\cal{N}F\text{-}\mathrm{Isoc}(X_0^\times/K)$ in such a way that $x_0^*$ gets identified with the forgetful functor.

By \cite[Proposition 9.9]{Laz16b} the affine group scheme over $K$ underlying $G^{\lcris}_{X_0,x_0}$, i.e. the affine group scheme obtained by forgetting $(\varphi,N)$-module structures, is simply $\pi_1^{\log\text{-}\mathrm{cris}}(X_0^\times/K^\times,x_0)$. Hence we may view this latter group scheme canonically as having the structure of a non-abelian $(\varphi,N)$-module. Hence the base change  $\pi_1^{\log\text{-}\mathrm{cris}}(X_0^\times/K^\times,x_0) \otimes_K \rk$ can be viewed as a non-abelian $\pn$-module over $\rk$. Alternatively, we may consider the non-abelian $\pn$-module $\pi_1^\rig(X/\rk,x)$ defined in \S\ref{recap}. The following is then the non-abelian analogue of Proposition \ref{plc}.

\begin{proposition} \label{fgssp} There is a natural isomorphism
\[ \pi_1^{\log\text{-}\mathrm{cris}}(X_0^\times/K^\times,x_0) \otimes_K \rk \cong \pi_1^\rig(X/\rk,x)\]
of non-abelian $\pn$-modules.
\end{proposition} 

\begin{proof} By definition there exists a non-abelian $\pn$-module $\pi_1^\rig(X/\ekd,x)$ such that 
\[ \pi_1^\rig(X/\ekd,x) \otimes_{\ekd} \rk \cong \pi_1^\rig(X/\rk,x). \]
Let $\cal{E}^+_K=W\pow{t}\otimes_W K$ be the ring of bounded power series over $K$, thus we have $\cal{E}_K^+\subset \ekd$ as the sub-ring of series with no terms in negative powers of $t$, we also have a natural `evaluate at $t=0$ map' $\cal{E}_K^+\rightarrow K$. Thanks to \cite[Proposition 3.1.4]{Tsu98} the category $\underline{\mathbf{M}\Phi}_{S_K}^{\nabla,\log} $ of log-$\pn$-modules over $\cal{E}_K^+$ is Tannakian, and it was shown in \cite[\S9]{Laz16b} (see in particular Propositions 9.6 and 9.9) that there exists a non-abelian log-$\pn$-module $\pi_1^{\lcris}(\cal{X}^\times /\cal{E}^+_K,x)$ over $\cal{E}_K^+$ such that 
\begin{align*}
 \pi_1^\rig(X/\ekd,x) &\cong \pi_1^{\lcris}(\cal{X}^\times/\cal{E}_K^+,x)\otimes_{\cal{E}_K^+} \cal{E}_K^\dagger \\
  \pi_1^{\log\text{-}\mathrm{cris}}(X_0^\times/K^\times,x_0)  &\cong \pi_1^{\lcris}(\cal{X}^\times/\cal{E}_K^+,x)\otimes_{\cal{E}_K^+} K.
\end{align*}
We can thus conclude by applying \cite[Proposition 5.34]{LP16}. \end{proof}

Now following the argument of the previous section, and given Theorem \ref{main1}, to prove Theorem \ref{main3} it suffices to show the following: suppose we are given a smooth, geometrically connected curve $C/k$, $c\in C(k)$ such that $F\cong \widehat{k(C)}_c$, and a semistable scheme $\cal{X}\rightarrow C$, smooth away from $c$, such that $X=\cal{X}_F$. Suppose further that $x$ comes from a section $s:C\rightarrow \cal{X}$. Let $U=C\setminus c$. Then we want to show that there exist local systems $\cal{F}_{\ell,k}$ on $U$ for all $\ell$ prime and $k\geq 1$ such that:
\begin{enumerate}\item for any closed point $y\in U$ the fibre of $\cal{F}_{\ell,k}$ at $y$ is isomorphic (as a $ W_{k(y)}$-representation) to the quotient of the universal enveloping algebra of $\mathrm{Lie}\;\pi_1^\ell(\cal{X}_y,s(y))$ by the $k$th power of its augmentation ideal;
\item the fibre of $\cal{F}_{\ell,k}$ at $c$ is isomorphic (as a $ W'_F$-representation) to the quotient of the universal enveloping algebra of $\mathrm{Lie}\;\pi_1^\ell(X,x)$ by the $k$th power of its augmentation ideal.
\end{enumerate}

For $\ell=p$ these follow from the results of \cite[\S3]{Laz15} and \cite[\S6]{Laz16b}. Specifically, in \cite[\S3]{Laz15} it was shown how to define, in the above situation, a pro-unipotent affine group scheme $\pi_1^\rig(\cal{X}_U/U,s)$ over the category $F\text{-}\mathrm{Isoc}(U/K)$, such that for any closed point $y\in U$ the fibre $\pi_1^\rig(\cal{X}_U/U,s)_y$ can be identified with the unipotent rigid fundamental group $\pi_1^\rig(\cal{X}_y/K,s(y))$ of the fibre $\cal{X}_y$. It was then proved in  \cite[Theorem 6.1]{Laz16b} that we have an isomorphism
\[ \pi_1^\rig(\cal{X}_U/U,s)_c \cong \pi_1^\rig(X/\ekd,x)\]
as non-abelian $\pn$-modules over $\ekd$, where $\pi_1^\rig(\cal{X}_U/U,s)_c$ denotes the `fibre' at $c$. Taking the quotients of the universal enveloping algebra of $\mathrm{Lie}\pi_1^\rig(\cal{X}_U/U,s)$ by the powers of the augmentation ideal then gives us the required $F$-isocrystals on $U/K$.

For $\ell\neq p$ these two follow from Grothendieck's homotopy exact sequence: we know that $\pi_1^\et(U,\bar\eta)$ acts on $\pi_1^\et(X_{\bar\eta},s(\bar \eta))$, hence on its pro-unipotent completion $\pi_1^\et(X_{\bar\eta},s(\bar \eta))_{\Q_\ell}$ and therefore on the universal enveloping algebra $\hat{\cal{U}}(\mathrm{Lie}\;\pi_1^\et(X_{\bar\eta},s(\bar \eta))_{\Q_\ell})$ of the associated Lie algebra (again, here $\bar\eta$ is any geometric generic point). Hence there is a natural action of $\pi_1^\et(U,\bar\eta)$ on $\hat{\cal{U}}(\mathrm{Lie}\;\pi_1^\et(X_{\bar\eta},s(\bar \eta))_{\Q_\ell})/\mathfrak{a}^k$ for any $k$ ($\mathfrak{a}$ being the augmentation ideal), and this gives a the required lisse $\Q_\ell$-sheaf on $U$.

The astute reader will object that there is a slight gap in the argument outlined above, in that while we can certainly assume that our variety is globally defined, we have not explained why we can assume the existence of a global section. Examining the proof of Proposition \ref{spread}, it suffices to observe that again the given formal section can be arbitrarily well approximated by a Henselian section, which must descend to some \'etale neighbourhood of $s$. While this will not agree with our given formal section, it will do so on the special fibre, which suffices by Propositions \ref{fgssl} and \ref{fgssp}.

\section{Curves in mixed characteristic}\label{cmc}

Finally, we will use the results of the previous section to deduce an $\ell$-independence result for the unipotent fundamental group of a semistable curve in mixed characteristic. So for this section $F$ will be assumed to be a mixed characteristic local field, with ring of integers $R$, and our main result is as follows.

\begin{theorem} Let $\cal{X}/R$ be a proper, semistable curve with generic fibre $X$, and $x\in \cal X(R)$. Then the conjecture $C_{\ell,w}(X,\hat{\cal U}/\mathfrak{a}^\bu)$ holds. 
\end{theorem}

\begin{proof} The basic idea is that $(\hat{\cal U}/\mathfrak{a}^k )_\ell$ (as a Weil--Deligne representation) only depends on the special fibre $X_0^\times$ of $\cal X$ as a log scheme, and since $X_0$ is a curve, this deforms to equicharacteristic. More concretely, when $\ell\neq p$ we consider the log-unipotent fundamental group $\pi_1^\et(X_0^{\times,\mathrm{tame}},x_0)_{\Q_\ell}$ of the log scheme $X_0^{\times,\mathrm{tame}}$ as a $G_F$-representation, then again applying \cite[Theorems 2.3, 2.5]{Lep13} we have
\[ \pi_1^\et(X_0^{\times,\mathrm{tame}},x_0)_{\Q_\ell} \cong \pi_1^\et(X_{F^\mathrm{sep}},x)_{\Q_\ell} \]
as $G_F$-representations. When $\ell=p$ we instead consider the log-crystalline fundamental group $\pi_1^{\log\text{-}\mathrm{cris}}(X_0^\times/K^\times,x_0)$ defined in \S\ref{uel} above. Thus by \cite[Theorem 1.4]{AIK15} or \cite[Theorem 10.2]{Laz16b} we have an isomorphism
\[ \mathbf{D}_\mathrm{st}(\pi_1^\et(X_{F^\mathrm{sep}},x)_{\Q_p}) \cong \pi_1^{\log\text{-}\mathrm{cris}}(X_0^\times/K^\times,x_0) \]
of $(\varphi,N)$-modules over $K$ (after forgetting the Hodge filtration on the former). 

But now since the log-smooth deformation theory of curves is unobstructed by \cite[Proposition 8.6]{Kat96}, we may deform $X_0^\times$ to characteristic $p$, in other words there exists a semistable family $\cal{Y}\rightarrow \spec{k\pow{t}}$ whose special fibre is isomorphic to $X_0^\times$ as a log scheme. Hence by applying Theorem \ref{main3} we know that the Weil--Deligne representations associated to the quotients of the universal enveloping algebras of $ \pi_1^\et(X_0^{\times,\mathrm{tame}},x_0)_{\Q_\ell}$ and $\pi_1^{\log\text{-}\mathrm{cris}}(X_0^\times/K^\times,x_0)$ are weakly independent of $\ell$, and we may conclude.
\end{proof}

\section{Corrigendum}\label{corrig}

It was pointed out to us by W. Zheng that the proof of Theorem \ref{main22} above is invalid. The problem is in the final step of the proof, where we showed that there was an exact sequence
\[ 0 \rightarrow H^{i+n}_\ell(X) \rightarrow H^{i+n}_\ell(X_0) \rightarrow H^{i+n}_\ell(X_1) \rightarrow \ldots \]
and claimed to deduce $\ell$-independence of $H^i_\ell(X)$ from $\ell$-independence of all the other terms $H^{i+n}_\ell(X_n)$. Of course, this deduction doesn't work, since there might be infinitely many such other terms. 

In their preprint \cite{LZ18} Q. Lu and W. Zheng provide (amongst other things) an alternative proof of this $\ell$-independence result, at least for $\ell\neq p$, see Theorem 1.4(2). In this corrigendum we will explain how to fix the proof of Theorem \ref{main22} by instead proving a stronger version of Corollary \ref{rg} where the semistable hypothesis is removed. In particular, this includes the case $\ell=p$. 

\subsection{Log structures}

We begin with a general result on semistable reduction and log schemes. Let $R$ be a complete DVR with perfect residue field $k$, $\pi$ a uniformiser for $R$, and and let $\mathcal{X}\rightarrow \spec{R}$ be a strictly semistable scheme. That is, $\mathcal{X}$ is Zariski locally \'etale over a standard semistable ring $R[x_1,\ldots,x_n]/(x_1\cdots x_r -\pi)$ for some $n,r$. There is a natural log structure $\mathcal{M}_\mathcal{X}$ on $\mathcal{X}$ given by functions invertible outside the special fibre $X$, and we let $\mathcal{M}_X$ denote the pull-back of this log structure to $X$. We will also write $X_i$ for the reduction of $\mathcal{X}$ modulo $\pi^{i+1}$, and $k^\times$ for $k$ equipped with the log structure pulled back from the canonical log structure $R^\times$ on $R$.

\begin{proposition}[Illusie, Nakayama \cite{Nak98} Appendix A.4] \label{prop: illusie-kato} If $\mathcal{X},\mathcal{X}'$ are strictly semistable schemes over $R$, and $g:X_1 \rightarrow X'_1$ is an isomorphism between their mod $\pi^2$-reductions, then $g$ induces a canonical isomorphism $g:(X,\mathcal{M}_X)\isomto (X',\mathcal{M}_{X'})$ of log schemes over $k^\times$.
\end{proposition}

\begin{proof}[Sketch of proof]
Use $g$ to identify $X_1$ and $X'_1$, and thus $X$ and $X'$. Let $\mathcal{M}_X$ and $\mathcal{M}'_X$ be the log structures on $X$ coming from $\mathcal{X}$ and $\mathcal{X}'$ respectively.

Near a closed point of $X$ let $X^{(1)},\ldots,X^{(r)}$ be the irreducible components of $X$, and pick $x_1,\ldots,x_r\in \mathcal{O}_\mathcal{X}$ such that $X^{(i)}=V(x_i)$. Similarly pick $x'_1,\ldots,x'_r\in \mathcal{O}_{\mathcal{X}'}$ such that $X^{(i)}=V(x'_i)$. Let $v\in \mathcal{O}_\mathcal{X}^*$ and $v'\in \mathcal{O}_{\mathcal{X}'}^*$ be such that $x_1\ldots x_r= v\pi$ and $x'_1\ldots x'_r= v'\pi$. Then in a neighbourhood of $p$ the morphisms $(\mathcal{X},\mathcal{M}_\mathcal{X})\rightarrow \spec{R^\times}$ and $(\mathcal{X}',\mathcal{M}_{\mathcal{X}'})\rightarrow \spec{R^\times}$ can be described by the following diagrams:
\[ \xymatrix{ \mathcal{M}_\mathcal{X} \cong \mathcal{O}_\mathcal{X}^* \oplus \N^r  \ar[rrrr]^-{(u,a_1,\ldots,a_r)\mapsto ux_1^{a_1}\ldots x_r^{a_r}} &&&&  \mathcal{O}_\mathcal{X} \\ \ar[u]^-{(\lambda,a)\mapsto (\lambda v^{-a} ,a,\ldots,a)} R^* \oplus \N \ar[rrrr]^-{(\lambda,a)\mapsto \lambda\pi^a} &&&& R. \ar[u]  } \]
\[ \xymatrix{ \mathcal{M}_{\mathcal{X}'} \cong \mathcal{O}_{\mathcal{X}'}^* \oplus \N^r  \ar[rrrr]^-{(u,a_1,\ldots,a_r)\mapsto ux_1'^{a_1}\ldots x_r'^{a_r}} &&&&  \mathcal{O}_{\mathcal{X}'} \\ \ar[u]^{(\lambda,a)\mapsto (\lambda v'^{-a},a,\ldots,a)} R^* \oplus \N \ar[rrrr]^-{(\lambda,a)\mapsto  \lambda \pi^a} &&&& R. \ar[u]  } \]
Pulling back to $k$, we see that the morphisms $(X,\mathcal{M}_X)\rightarrow \spec{k^\times}$ and $(X,\mathcal{M}'_X) \rightarrow \spec{k^\times}$ can be described by the diagrams
\[ \xymatrix{ \mathcal{M}_X \cong \mathcal{O}_X^* \oplus \N^r  \ar[rrrr]^-{(u,a_1,\ldots,a_r)\mapsto ux_1^{a_1}\ldots x_r^{a_r}} &&&&  \mathcal{O}_X \\ \ar[u]^-{(\lambda,a)\mapsto (\lambda v^{-a} ,a,\ldots,a)} k^* \oplus \N \ar[rrrr]^-{(\lambda,a)\mapsto \lambda0^a} &&&& k. \ar[u]  } \]
and 
\[ \xymatrix{ \mathcal{M}'_X \cong \mathcal{O}_X^* \oplus \N^r  \ar[rrrr]^-{(u,a_1,\ldots,a_r)\mapsto ux_1'^{a_1}\ldots x_r'^{a_r}} &&&&  \mathcal{O}_X \\ \ar[u]^{(\lambda,a)\mapsto (\lambda v'^{-a},a,\ldots,a)} k^* \oplus \N \ar[rrrr]^-{(\lambda,a)\mapsto  \lambda 0^a} &&&& k \ar[u]  } \]
respectively, again in a neighbourhood of $p$. Since $V(x_i)=V(x_i')$ inside $X_1$, we must have $x_i=u_ix'_i$ for some $u_i \in \mathcal{O}^*_{X_1}$, and so we can define an isomorphism
\[ \mathcal{M}_X \isomto \mathcal{M}'_X \]
of log structures by mapping
\[ (u,a_1,\ldots,a_r) \mapsto (uu_1^{a_1}\ldots u_r^{a_r},a_1,\ldots,a_r). \]
This is checked to be a morphism of log structures over $k^\times$ by using the above local descriptions. Note that any other choice $u'_i$ must satisfy $(u_i-u'_i)x'_i=0$ in $\mathcal{O}_{X_1}$, and hence we must have $u_i-u_i'\in (\pi)$. In particular, the above isomorphism doesn't depend on the choice of $u_i$. By a similar argument, neither does it depend on the choice of $x_i$ and $x_i'$, and so it glues to give a global isomorphism $(X,\mathcal{M}_X)\cong (X,\mathcal{M}'_X)$ of log schemes over $k^\times$.
\end{proof}

We will need to extend this result to cover morphisms between strictly semistable schemes over different bases. So suppose that $R\rightarrow S$ is a finite morphism of complete DVR's, with induced residue field extension $k\rightarrow k_S$. Let $\pi_S$ be a uniformiser for $S$, and let $e=v_{\pi_S}(\pi)$. We do not assume that the induced extension $Q(R)\rightarrow Q(S)$ of fraction fields is separable.

Suppose that we have strictly semistable schemes $\mathcal{X},\mathcal{X}'$ over $R$ and $\mathcal{Y},\mathcal{Y}'$ over $S$, and a pair of commutative diagrams
\[ \xymatrix{ \mathcal{Y} \ar[r]^f \ar[d] & \mathcal{X} \ar[d] & \mathcal{Y}'\ar[r]^{f'}\ar[d] & \mathcal{X}' \ar[d] \\
\spec{S} \ar[r] & \spec{R} & \spec{S} \ar[r] & \spec{R}.  } \]
As before, let us write $Y_j$ for the reduction of $\mathcal{Y}$ modulo $\pi_S^{j+1}$. Suppose that we have isomorphisms
\[ g_Y: Y_{e} \isomto Y'_{e},\;\;g_X :X_1\isomto X'_1 \]
of $S$- and $R$-schemes respectively such that the diagram
\[ \xymatrix{ Y_{e} \ar[r]^f\ar[d]_{g_Y} & X_1 \ar[d]^{g_X} \\
 Y'_{e} \ar[r]^{f'}\ar[d] & X'_1 \ar[d] \\
 \spec{S} \ar[r] & \spec{R} } \]
commutes. Then by Proposition \ref{prop: illusie-kato} we obtain isomorphisms
\[ g_Y: (Y,\mathcal{M}_{Y}) \isomto (Y',\mathcal{M}_{Y'}) \]
of log schemes over $k_S^\times$, as well as 
\[ g_X: (X,\mathcal{M}_{X}) \isomto (X',\mathcal{M}_{X'}) \]
of log schemes over $k^\times$. The above commutative diagrams of strictly semistable schemes induce commutative diagrams
\[ \xymatrix{ (Y,\mathcal{M}_{Y}) \ar[r]^f \ar[d] & (X,\mathcal{M}_{X}) \ar[d] & (Y',\mathcal{M}_{Y'})\ar[r]^{f'}\ar[d] & (X,\mathcal{M}_{X'}) \ar[d] \\
\spec{k_S^\times} \ar[r] & \spec{k^\times} & \spec{k_S^\times} \ar[r] & \spec{k}  } \]
of log schemes. Note that the morphism of punctured points along the bottom of each square is given by
\begin{align*}
k^* \oplus \N &\rightarrow k_S^* \oplus \N \\
(\lambda,a) &\mapsto (\lambda u^a,ea)
\end{align*}
where $u\in S^*$ is such that $\pi=u\pi_S^e$.

\begin{proposition} \label{prop: mod 2 functorial} The diagram 
\[ \xymatrix{  (Y,\mathcal{M}_{Y}) \ar[r]^f\ar[d]^{g_Y}_\cong & (X,\mathcal{M}_{X}) \ar[d]^{g_X}_\cong \\ (Y',\mathcal{M}_{Y'}) \ar[r]^{f'} & (X',\mathcal{M}_{X'})} \]
of log schemes commutes.
\end{proposition}

\begin{proof}
Let us use $g$ to identify $Y_{e}=Y'_{e}$ and $Y=Y'$, and let $\mathcal{M}_Y$ and $\mathcal{M}'_Y$ be the log structures on $Y$ coming from $\mathcal{Y}$ and $\mathcal{Y}'$ respectively. Similarly identify $X_1=X'_1$ and $X=X'$, and let $\mathcal{M}_X$ and $\mathcal{M}'_X$ be the log structures on $X$ coming from $\mathcal{X}$ and $\mathcal{X}'$ respectively. 

Locally on $X$ and $Y$, choose functions $y_1,\ldots,y_s \in \mathcal{O}_\mathcal{Y}$, $y'_1,\ldots,y'_s \in \mathcal{O}_{\mathcal{Y}'}$ cutting out the irreducible components of $Y$, and functions $x_1,\ldots,x_r\in \mathcal{O}_{\mathcal{X}}$ and  $x'_1,\ldots,x'_r\in \mathcal{O}_{\mathcal{X}'}$ cutting out the irreducible components of $X$. Write
\[ f^*(x_i) = \alpha_iy_1^{d_{i1}}\ldots y_s^{d_{is}},\;\;  {f'}^*(x'_i) = \alpha'_i{y_1'}^{d'_{i1}}\ldots {y_s'}^{d'_{is}},\]
since both $d_{ij}$ and $d_{ij}'$ are given by the multiplicity of the $j$th irreducible component of $Y$ in the scheme theoretic preimage of the $i$th irreducible component of $X$ inside $Y_{e}$, we must have $d_{ij}=d'_{ij}$. Moreover, since $V(f^*(x_i))\subset V(\pi_S^e)=V(y_1^e\ldots y_s^e)$ we must have $d_{ij}\leq e$ for all $i,j$.

Now choose $u_i\in \mathcal{O}_{X_1}^*$ such that $x_i=u_ix'_i$, and $v_j\in \mathcal{O}_{Y_{e}}^*$ such that $y_j=v_jy_j'$. Then the isomorphisms of log structures induced by $g_Y$ and $g_X$ are given by
\begin{align*}
\mathcal{M}_Y =  \mathcal{O}_Y^*\oplus \N^s &\rightarrow \mathcal{M}'_Y = \mathcal{O}_Y^*\oplus \N^s  \\
(v,b_1,\ldots,b_s) &\mapsto (vv_1^{b_1}\ldots v_s^{b_s},b_1,\ldots,b_s)
\end{align*} 
and
\begin{align*}
\mathcal{M}_X =  \mathcal{O}_X^*\oplus \N^r &\rightarrow  \mathcal{M}'_X =\mathcal{O}_X^*\oplus \N^r   \\
(u,a_1,\ldots,a_r) &\mapsto (uu_1^{a_1}\ldots u_r^{a_r},a_1,\ldots,a_r)
\end{align*} 
respectively, and the morphisms $\mathcal{M}_X\rightarrow \mathcal{M}_{Y}$ and $\mathcal{M}_{X}'\rightarrow \mathcal{M}_{Y}'$ are defined by 
\[
(u,a_1,\ldots,a_r)\mapsto (f^*(u)\alpha_1^{a_1}\ldots\alpha_r^{a_r},\sum_ i d_{i1}a_i,\ldots, \sum_i d_{is}a_i)
\]
and 
\[ (u,a_1,\ldots,a_r)\mapsto (f^*(u){\alpha_1'}^{a_1}\ldots{\alpha_r'}^{a_r},\sum_ i d_{i1}a_i,\ldots, \sum_i d_{is}a_i) \]
respectively. Hence in the diagram
\[ \xymatrix{ \mathcal{M}_X \ar[r]^{g_X} \ar[d]_f & \mathcal{M}'_{X} \ar[d] ^f \\ \mathcal{M}_Y \ar[r]^{g_Y} & \mathcal{M}'_{Y}  }\]
the composite $f\circ g_X$ is given by
\[ (u,a_1,\ldots,a_r) \mapsto (f^*(u)(\alpha'_1f^*(u_1))^{a_1}\ldots(\alpha'_rf^*(u_r))^{a_r},\sum_ i d_{i1}a_i,\ldots, \sum_i d_{is}a_i)\]
and the composite $g_Y\circ f$ is given by
\[ (u,a_1,\ldots,a_r) \mapsto (f^*(u) (\alpha_1v_1^{d_{11}}\ldots v_s^{d_{1s}})^{a_1}\ldots (\alpha_rv_1^{d_{r1}}\ldots v_s^{d_{rs}})^{a_r},\sum_ i d_{i1}a_i,\ldots, \sum_i d_{is}a_i ) .\]
We thus need to show that $\alpha'_if^*(u_i)=\alpha_iv_1^{d_{i1}}\ldots v_s^{d_{is}}$ in $\mathcal{O}_{Y}^*$ for all $i$. But now we write
\[ \alpha_i{y}_1^{d_{i1}}\ldots {y}_s^{d_{is}} =f^*(x_i)=f^*(u_ix'_i) = f^*(u_i)\alpha'_i{y_1'}^{d_{i1}}\ldots {y_s'}^{d_{is}}  \]
in $\mathcal{O}_{Y_{e}}$ and so deduce that
\[  \alpha_iv_1^{d_{i1}}\ldots v_s^{d_{is}}{y_1'}^{d_{i1}}\ldots {y_s'}^{d_{is}} = f^*(u_i)\alpha'_i{y_1'}^{d_{i1}}\ldots {y_s'}^{d_{is}}.  \]
We deduce that the difference $\beta_i=\alpha_i'f^*(u_i)-\alpha_iv_1^{d_{i1}}\ldots v_s^{d_{is}}$ annihilates ${y_1'}^{d_{i1}}\ldots {y_s'}^{d_{is}}$ inside $\mathcal{O}_{Y_{e}}$, and since each $d_{ij}\leq e$ we deduce that in fact $\beta_i$ annihilates $\pi_S^e$, and therefore must lie in $(\pi_S)$. Hence $\beta_i=0$ in $\mathcal{O}_Y$ and the proof is complete. 
\end{proof}

\subsection{Functoriality of comparison isomorphisms}

We will also need to know that the comparison isomorphisms, Propositions \ref{llc} and \ref{plc}, are compatible with morphisms of semistable schemes over different bases. So let us suppose that we are again in the above set-up, where we have a commutative diagram
\[ \xymatrix{ \mathcal{Y} \ar[r]\ar[d] & \mathcal{X} \ar[d] \\ \spec{S} \ar[r] & \spec{R} }  \]
of strictly semistable schemes $\mathcal{Y}$ and $\mathcal{X}$ over $S$ and $R$ respectively, with $S$ the integral closure of $R$ in some finite extension of its fraction field. Let us assume that $R$, and hence $S$, is of equicharacteristic $p>0$, with fraction fields $F$ and $F_S$ respectively, whose absolute Galois groups we will denote by $G_F$ and $G_{F_S}$. Fix an embedding $F^\mathrm{sep} \hookrightarrow F_S^\mathrm{sep}$ of separable closures, note that this sends $F^\mathrm{tame}$ into $F_S^\mathrm{tame}$ and induces an injective homomorphism $G_{F_S} \rightarrow G_F$ with finite cokernel.

Let $\mathcal{X}^\times$ and $\mathcal{Y}^\times$ denote these semistable schemes endowed with their canonical log structures, and  $X^\times$ and $Y^\times$ the corresponding log special fibres. We therefore have a commutative diagram
\[ \xymatrix{ Y^\times \ar[r]\ar[d] & X^\times \ar[d] \\ \spec{k_S^\times} \ar[r] & \spec{k^\times} } \]
of log schemes. For every finite subextension $F\subset L \subset F^\mathrm{tame}$, let $X^{\times}_L$ denote the corresponding base change of $X^\times$, and $X^{\times,\mathrm{tame}}$ the inverse limit of the \'etale topoi of all such $X^{\times}_L$; we have $Y^{\times,\mathrm{tame}}$ defined entirely similarly. Via the embedding $F^\mathrm{tame} \hookrightarrow F_S^\mathrm{tame}$ this induces a $G_{F_S}$-equivariant morphism of topoi
\[ Y^{\times,\mathrm{tame}} \rightarrow X^{\times,\mathrm{tame}} \]
and hence a $G_{F_S}$-equivariant morphism
\[ H^i_\et(X^{\times,\mathrm{tame}},\Q_\ell) \rightarrow H^i_\et(Y^{\times,\mathrm{tame}},\Q_\ell) \]
in cohomology, for any $\ell\neq p$. On the other hand we have a natural $G_{F_S}$-equivariant map
\[ H^i_\et(\mathcal{X}\times_R F^\mathrm{sep} ,\Q_\ell) \rightarrow H^i_\et(\mathcal{Y}\times_S F_S^\mathrm{sep} ,\Q_\ell), \]
and by \cite[Proposition 4.2]{Nak98} equivariant isomorphisms
\begin{align*} H^i_\et(X^{\times,\mathrm{tame}},\Q_\ell) &\isomto H^i_\et(\mathcal{X}\times_R F^\mathrm{sep} ,\Q_\ell) \\
H^i_\et(Y^{\times,\mathrm{tame}},\Q_\ell) &\isomto H^i_\et(\mathcal{Y}\times_S F_S^\mathrm{sep} ,\Q_\ell).
\end{align*}

\begin{proposition} \label{prop: comp functorial 1} The diagram
\[ \xymatrix{ H^i_\et(X^{\times,\mathrm{tame}},\Q_\ell) \ar[r]\ar[d]  & H^i_\et(\mathcal{X}\times_R F^\mathrm{sep} ,\Q_\ell) \ar[d] \\ H^i_\et(Y^{\times,\mathrm{tame}},\Q_\ell) \ar[r] & H^i_\et(\mathcal{Y}\times_S F_S^\mathrm{sep} ,\Q_\ell)  } \]
commutes.
\end{proposition}

\begin{proof}
Consider the commutative diagram
\[ \xymatrix{ Y^{\times,\mathrm{tame}} \ar[r]^-{i_Y} \ar[d]_f & \mathcal{Y}^{\times,\mathrm{tame}} \ar[d]_f  & \mathcal{Y}_{F_S^\mathrm{sep}} \ar[d]^f \ar[l]_{j_\mathcal{Y}} \\ X^{\times,\mathrm{tame}} \ar[r]^-{i_X} & \mathcal{X}^{\times,\mathrm{tame}} &  \ar[l]_{j_\mathcal{X}}\mathcal{X}_{F^\mathrm{sep}}  } \]
of topoi as in \cite[\S3]{Nak98}, where $\mathcal{Y}^{\times,\mathrm{tame}}$ and $\mathcal{X}^{\times,\mathrm{tame}}$ are defined by `base change' along $F_S\rightarrow F_S^\mathrm{tame}$ and $F\rightarrow F^\mathrm{tame}$ respectively. Then the isomorphism
\[ H^i_\et(Y^{\times,\mathrm{tame}},\Q_\ell) \isomto H^i_\et(\mathcal{Y}\times_S F_S^\mathrm{sep} ,\Q_\ell) \]
is given as the composite
\[ H^i_\et(Y^{\times,\mathrm{tame}},\Q_\ell) \isomfrom H^i_\et(\mathcal{Y}^{\times,\mathrm{tame}},\Q_\ell) \rightarrow H^i_\et(\mathcal{Y}\times_S F_S^\mathrm{sep} ,\Q_\ell) \]
using the proper base change theorem in log-\'etale cohomology \cite[Theorem 5.1]{Nak97}, and there is a similar statement for $\mathcal{X}$. The claim then follows simply from commutativity of the above diagram of log schemes.
\end{proof}

We will also need a version of this result for $p$-adic cohomology. Write $W=W(k)$, $W_S=W(k_S)$, let $K=W[1/p]$, $K_S=W_S[1/p]$, and let $\mathcal{R}_K\supset \mathcal{E}^\dagger_K\subset \mathcal{E}_K$, and $\mathcal{R}_{K_S}\supset \mathcal{E}^\dagger_{K_S}\subset \mathcal{E}_{K_S}$ denote copies of the Robba ring, the bounded Robba ring and the Amice ring over $K$ and $K_S$ respectively. Lift the extension $F\rightarrow F_S$ to a finite flat morphism $\mathcal{E}^\dagger_K\rightarrow \mathcal{E}^\dagger_{K_S}$ which extends to finite flat morphisms $\mathcal{R}_K\rightarrow \mathcal{R}_{K_S}$ and $\mathcal{E}_K\rightarrow \mathcal{E}_{K_S}$. Then as above, the morphism of log schemes $Y^\times\rightarrow X^\times$ induces a morphism
\[ H^i_{\log\text{-}\mathrm{cris}}(X^{\times}/K^\times) \rightarrow H^i_{\log\text{-}\mathrm{cris}}(Y^{\times}/K^\times_S) \]
in log crystalline cohomology, and the morphism $\mathcal{Y}_{F_S}\rightarrow \mathcal{X}_F$ induces a morphism
\[ H^i_\rig(\mathcal{X}_F/\mathcal{R}_K) \rightarrow H^i_\rig(\mathcal{Y}_{F_S}/\mathcal{R}_{K_S}) \]
in Robba-ring valued rigid cohomology. Then following Proposition \ref{plc} we can construct isomorphisms\begin{align*} H^i_{\log\text{-}\mathrm{cris}}(X^{\times}/K^\times) \otimes_K \mathcal{R}_K &\isomto H^i_\rig(\mathcal{X}_F/\mathcal{R}_K) \\
H^i_{\log\text{-}\mathrm{cris}}(Y^{\times}/K^\times_S) \otimes_{K_S} \mathcal{R}_{K_S} &\isomto H^i_\rig(\mathcal{Y}_{F_S}/\mathcal{R}_{K_S})
\end{align*}
as follows. Let $t$ denote a co-ordinate on $\mathcal{E}_K^\dagger$ and $t_S$ a co-ordinate on $\mathcal{E}_{K_S}^\dagger$ such that $t\in W_S\pow{t_S}$. Write $S_K=K\otimes W\pow{t}$ and $S_{K_S}=K_S\otimes W_S\pow{t_S}$. Equip $W\pow{t}$ (resp. $W_S\pow{t_S})$ with the log structure defined by the ideal $(t)\subset W\pow{t}$ (resp. $(t_S)\subset W\pow{t_S}$) and define the log-crystalline cohomology groups
\begin{align*}
H^i_{\log\text{-}\mathrm{cris}}(\mathcal{X}^{\times}/S_{K}) &:= H^i_{\log\text{-}\mathrm{cris}}(\mathcal{X}^{\times}/W\pow{t}) \otimes_{\Z} \Q   \\
H^i_{\log\text{-}\mathrm{cris}}(\mathcal{Y}^{\times}/S_{K_S}) &:= H^i_{\log\text{-}\mathrm{cris}}(\mathcal{Y}^{\times}/W_S\pow{t_S}) \otimes_{\Z} \Q,  
\end{align*}
these are naturally endowed with the extra structure of $\log\text{-}\pn$-modules over $S_K$ and $S_{K_S}$ respectively. Moreover, we have isomorphisms of $\varphi$-modules
\begin{align*}
H^i_{\log\text{-}\mathrm{cris}}(\mathcal{X}^{\times}/S_{K}) \otimes_{S_K,t\mapsto0} K &\isomto   H^i_{\log\text{-}\mathrm{cris}}(Y^{\times}/K_S^\times) \\
H^i_{\log\text{-}\mathrm{cris}}(\mathcal{Y}^{\times}/S_{K_S}) \otimes_{S_{K_S},t_S\mapsto0} K_S &\isomto H^i_{\log\text{-}\mathrm{cris}}(Y^{\times}/K_S^\times) ,  
\end{align*}
by smooth and proper base change in log-crystalline cohomology, as well as isomorphisms of $\pn$-modules
 \begin{align*}
H^i_{\log\text{-}\mathrm{cris}}(\mathcal{X}^{\times}/S_{K}) \otimes_{S_K} \mathcal{R}_K &\isomto  H^i_\rig(\mathcal{X}_F/\mathcal{R}_K) \\
H^i_{\log\text{-}\mathrm{cris}}(\mathcal{Y}^{\times}/S_{K_S}) \otimes_{S_{K_S}} \mathcal{R}_{K_S} &\isomto H^i_\rig(\mathcal{Y}_{F_S}/\mathcal{R}_{K_S}) , \end{align*}
by \cite[Proposition 5.45]{LP16}. It therefore follows from the logarithmic form of Dwork's trick \cite[Corollary 17.2.4]{Ked10} that the $\pn$-modules $H^i_\rig(\mathcal{X}_{F}/\mathcal{R}_{K})$ and $H^i_\rig(\mathcal{Y}_{F_S}/\mathcal{R}_{K_S})$ are unipotent, that there are isomorphisms
\begin{align*}
\left( H^i_\rig(\mathcal{X}_{F}/\mathcal{R}_{K})[\log t]\right)^{\nabla=0} &\cong H^i_{\log\text{-}\mathrm{cris}}(X^{\times}/K^\times)   \\
\left(H^i_\rig(\mathcal{Y}_{F_S}/\mathcal{R}_{K_S})[\log t_S] \right)^{\nabla=0} &\cong H^i_{\log\text{-}\mathrm{cris}}(Y^{\times}/K_S^\times)    
\end{align*}
and moreover the connection $\nabla$ on the rigid cohomology groups appearing on the LHS can be completely recovered from the monodromy operator $N$ on the RHS. This allows us to construct isomorphisms of $\pn$-modules 
\begin{align*} H^i_{\log\text{-}\mathrm{cris}}(X^{\times}/K^\times) \otimes_K \mathcal{R}_K &\isomto H^i_\rig(\mathcal{X}_F/\mathcal{R}_K) \\
H^i_{\log\text{-}\mathrm{cris}}(Y^{\times}/K^\times_S) \otimes_{K_S} \mathcal{R}_{K_S} &\isomto H^i_\rig(\mathcal{Y}_{F_S}/\mathcal{R}_{K_S})
\end{align*}
where the LHS is endowed a natural connection coming from $N$, for more details see for example \cite[\S3.2]{Mar08}.

\begin{proposition} \label{prop: comp functorial 2} The diagram
\[ \xymatrix{ H^i_{\log\text{-}\mathrm{cris}}(X^{\times}/K^\times) \otimes_K \mathcal{R}_K \ar[r]\ar[d]  & H^i_\rig(\mathcal{X}_F/\mathcal{R}_K) \ar[d] \\ H^i_{\log\text{-}\mathrm{cris}}(Y^{\times}/K^\times_S) \otimes_{K_S} \mathcal{R}_{K_S} \ar[r] & H^i_\rig(\mathcal{Y}_{F_S}/\mathcal{R}_{K_S})  } \]
commutes.
\end{proposition}

\begin{proof}
Given the construction of the horizontal isomorphisms outlined above, it suffices to show that the diagram
\[ \xymatrix{  H^i_{\log\text{-}\mathrm{cris}}(\mathcal{X}^{\times}/S_K) \ar[r] \ar[d] &  H^i_{\log\text{-}\mathrm{cris}}(X^{\times}/K^\times)  \ar[d]  \\H^i_{\log\text{-}\mathrm{cris}}(\mathcal{Y}^{\times}/S_{K_S}) \ar[r] & 
H^i_{\log\text{-}\mathrm{cris}}(Y^{\times}/K^\times_S)     } \] 
of log-crystalline cohomology groups commutes, which as in Proposition \ref{prop: comp functorial 1} simply follows from functoriality of log-crystalline cohomology.
\end{proof}

\subsection{Cohomology and global approximations}

Now suppose that $k$ is a finite field, $F=k\lser{t}$, and $X/F$ is a smooth and proper variety. 

\begin{definition} We say that $X$ is globally defined if there exist a smooth curve $C/k$, a $k$-valued point $c\in C(k)$, a smooth and proper morphism $\mathbf{X} \rightarrow ( C \setminus \{c\})$ and an isomorphism $F \cong \widehat{k(C)}_c$ such that $\mathbf{X}_{F} \cong X$. 
\end{definition}

We will prove the following strengthened version of Corollary \ref{rg}.

\begin{theorem} \label{theo: fix} For any smooth and proper variety $X/F$ there exists a globally defined smooth and proper variety $Z/F$ such that
\[ H^i_\ell(X) \cong H^i_\ell(Z) \]
for all $\ell$ (including $\ell=p$).
\end{theorem}

Once we have shown this, the proof of Theorem \ref{main22} can them be completed using Proposition \ref{gllo}, exactly as in the proof of Theorem \ref{main2}.

To prove Theorem \ref{theo: fix}, first of all choose a proper and flat model $\mathcal{X}$ for $X$ over the ring of integers $\mathcal{O}_F$. By \cite[Theorem 6.5]{dJ96} we may choose an alteration $\mathcal{X}_0\rightarrow \mathcal{X}$ and a finite extension $F_0/F$ such that $\mathcal{X}_0$ is strictly semistable over $\mathcal{O}_{F_0}$. 

Next, we take the fibre product $\mathcal{X}_0\times_\mathcal{X} \mathcal{X}_0$, and let $\mathcal{X}_1'$ denote the disjoint union of the reduced, irreducible components of $\mathcal{X}_0\times_\mathcal{X} \mathcal{X}_0$ which are flat over $\mathcal{O}_{F_0}$, or equivalently which map surjectively to $\spec{\mathcal{O}_{F_0}}$. Once more applying \cite[Theorem 6.5]{dJ96} to each of the connected components of $\mathcal{X}_1'$ in turn enables us to produce
\begin{itemize}
\item a 2-truncated augmented simplicial scheme
\[ \mathcal{X}_1 \rightrightarrows \mathcal{X}_0 \rightarrow \mathcal{X} \]
which is a proper hypercover after base changing to $F$;
\item a collection $F_{1,1},\ldots,F_{1,s}$ of finite field extensions of $F_0$;
\end{itemize}
such that $\mathcal{X}_1$ is a disjoint union of schemes $\mathcal{X}_{1,j}$, for $1\leq j\leq s$, proper and strictly semistable over $\spec{\mathcal{O}_{F_{1,j}}}$.

Let $k_0$ denote the residue field of $F_0$, $k_{1,j}$ the residue field of $F_{1,j}$, and consider the intermediate extensions
\[ F\subset F^\mathrm{un}_0 \subset F_0^s \subset F_0 \subset F_{1,j}^\mathrm{un} \subset F_{1,j}^s \subset F_{1,j}  \]
where $F^\mathrm{un}_0/F$ and $F^\mathrm{un}_{1,j}/F_0$ are separable and unramified, $F^s_0/F^\mathrm{un}_0$ and $F^s_{1,j}/F^\mathrm{un}_{1,j}$ are separable and totally ramified, and $F_0/F_0^s$ and $F_{1,j}/F_{1,j}^s$ are totally inseparable, of degree $p^{d_0}$ and $p^{d_{1,j}}$ respectively. Let $t$ denote a uniformiser for $F$, $t_0$ one for $F^s_0$, and let $P_0$ be the minimal polynomial of $t_0$ over $F_0^\mathrm{un}$. Then $t_0':=t_0^{1/p^{d_0}}$ is a uniformiser for $\mathcal{O}_{F_0}$. Similarly, let $t_{1,j}$ be a uniformiser for $F^s_{1,j}$, and $P_{1,j}$ the minimal polynomial of $t_{1,j}$ over $F_{1,j}^\mathrm{un}$. Then $t_{1,j}':=t_{1,j}^{1/p^{d_{1,j}}}$ is a uniformiser for $\mathcal{O}_{F_{1,j}}$.

Now choose a finitely generated sub-$k$-algebra $R\subset \mathcal{O}_{F}$, containing $t$, such that there exists a proper, flat scheme $\mathcal{Y}\rightarrow \spec{R}$ whose base change to $\mathcal{O}_F$ is exactly $\mathcal{X}$. By \cite[Theorem 10.1]{Spi99}, we may at any point increase $R$ to ensure that it is in fact smooth over $k$. Next, enlarge $R$ so that $R_0^\mathrm{un}:= R\otimes_k k_0 \subset \mathcal{O}_{F_0^\mathrm{un}}$ contains all the coefficients of the minimal (Eisenstien) polynomial $P_0$ of $t_0$, and let $R_0^s$ denote the corresponding finite flat extension $R_0^\mathrm{un}[x]/(P_0)$ of $R_0^\mathrm{un}$. We can thus consider $R_0^s \subset \mathcal{O}_{F_0^s}$ as a subring containing $t_0$, we set $R_0=R_0^s[t_0']$. Hence we have $ R_0 \subset \mathcal{O}_{F_0}$ such that
\[ R_0 \otimes_R \mathcal{O}_{F} \isomto \mathcal{O}_{F_0}.\]
Note also that $R_0$ is finite and flat over $R$, after localising $R$ within $\mathcal{O}_F$ we may in fact assume that $R_0$ is finite free over $R$.

Next we enlarge $R$ so that there exists a proper and flat morphism $\mathcal{Y}_0\rightarrow \spec{R_0}$ whose base change to $\mathcal{O}_{F_0}$ is $\mathcal{X}_0$. Again, by further enlarging $R$ we may in addition assume that the map $\mathcal{X}_0\rightarrow \mathcal{X}$ arises from a proper surjective map
\[ \mathcal{Y}_0\rightarrow \mathcal{Y}  \]
of $R$-schemes, and moreover that there exists an open cover of $\mathcal{Y}_0$ by schemes which are \'etale over $R_0[x_1,\ldots,x_n]/(x_1\ldots x_r -t'_0)$ for some $n,r$. In other words, $\mathcal{Y}_0$ is ``strictly $t_0'$-semistable''.

We now repeat this process to produce further finite free extensions $R_0 \rightarrow R_{1,j}^\mathrm{un}\rightarrow R_{1,j}^s\rightarrow R_{1,j}$ for all $j$, and an injection $R_{1,j}\subset \mathcal{O}_{F_{1,j}}$ containing the image of $t_{1,j}'$ such that
\[ R_{1,j} \otimes_R \mathcal{O}_{F} \isomto \mathcal{O}_{F_{1,j}}.\]
We can also find proper, strictly $t_{1,j}'$-semistable schemes $\mathcal{Y}_{1,j}\rightarrow \spec{R_{1,j}}$ whose base change to $\mathcal{O}_{F_{1,j}}$ is $\mathcal{X}_{1,j}$, so that setting $\mathcal{Y}_1:=\coprod_j \mathcal{Y}_{1,j}$ (and again, possibly increasing $R$), we obtain a 2-truncated augmented simplicial scheme
\[\mathcal{Y}_1\rightrightarrows \mathcal{Y}_0 \rightarrow \mathcal{Y}  \]
which becomes a proper hypercover over a dense open subscheme of $\spec{R}$, and whose base change to $\mathcal{O}_F$ is exactly our original 2-truncated augmented simplicial scheme
\[ \mathcal{X}_1\rightrightarrows \mathcal{X}_0 \rightarrow \mathcal{X}. \]
Let $\iota:R \hookrightarrow \mathcal{O}_F$ denote the canonical inclusion, and $\iota^*:\spec{\mathcal{O}_F}\rightarrow \spec{R}$ the induced morphism of schemes. Note that since $\iota^*$ maps the generic point of $\spec{\mathcal{O}_F}$ to that of $\spec{R}$, the map $\mathcal{Y}\rightarrow \spec{R}$ is generically smooth. We may thus choose an open subset $U\subset \spec{R}$ such that $\mathcal{Y}_U \rightarrow U$ is smooth, and such that the base change of $[\mathcal{Y}_1\rightrightarrows \mathcal{Y}_0 \rightarrow \mathcal{Y}]$ to $U$ is a proper hypercover.

\begin{lemma} \label{lemma: approx} For any $n\geq 0$ there exists a smooth curve $C/k$, a rational point $c\in C(k)$, a uniformiser $t_c$ at $c$, and a locally closed immersion $C\rightarrow \spec{R}$ such that $C\setminus \{c\} \subset U$, and the induced map
\[ \spec{\mathcal{O}_{C,c}/\mathfrak{m}_c^n} \rightarrow \spec{R} \]
agrees with the modulo $t^n$-reduction of $\iota^*$ via the isomorphism
\[ \widehat{\mathcal{O}}_{C,c} \isomto \mathcal{O}_F \]
sending $t_c$ to $t$.
\end{lemma} 

\begin{proof}
Since $R$ is smooth, we may choose \'etale co-ordinates around the image $\iota^*(s)$ of the closed point of $\spec{\mathcal{O}_F}$ under $\iota^*$. This induces an \'etale map $\spec{R}\rightarrow \A^n_k$ for some $n$, and it is a simple exercise to prove the corresponding claim for $\A^n_k$. We then just take the pull-back to $\spec{R}$.
\end{proof}

The canonical inclusion $\iota$ induces similar inclusions 
\[ \iota_0^{\#}: R_0^{\#} \hookrightarrow R_0^{\#} \otimes_R \mathcal{O}_F = \mathcal{O}_{F_0^{\#}} \]
for $\# \in \{\mathrm{un},s,\emptyset\}$, as well as 
\[ \iota_{1,j}^{\#}: R_{1,j}^{\#} \hookrightarrow R_{1,j}^{\#} \otimes_R \mathcal{O}_F = \mathcal{O}_{F_{1,j}^{\#}} \]
for all $j$, and again for $\# \in \{\mathrm{un},s,\emptyset\}$. We will need the following form of Krasner's lemma \cite[\href{https://stacks.math.columbia.edu/tag/0BU9}{Section 0BU9}]{stacks}.

\begin{lemma} \label{lemma: krasner} Let $K$ be a local field, with ring of integers $\mathcal{O}_K$, and let $P(x)$ be an Eisenstein polynomial over $\mathcal{O}_K$. Let $L$ be the corresponding finite totally ramified extension, and let $\alpha$ be a root of $P$ in $L$. Then for any $m\geq1$ there exists an $n\geq2$ such that any $Q(x)\in \mathcal{O}_K[x]$ congruent to $P$ modulo $\mathfrak{m}_K^n$ is Eisenstein, and $L$ contains a root $\beta$ of $Q$ such that $L=K(\beta)$ and $\alpha\equiv \beta$ modulo $\mathfrak{m}_L^m$.
\end{lemma}

We will use this as follows: given $n_1\geq \max_j\{[F_{1,j}:F]\}$ Lemma \ref{lemma: krasner} shows that there exists some $n_0\geq \max\left\{2,[F_0:F] \right\}$ such that any polynomial $Q_{1,j}$ with coefficients in $\mathcal{O}_{F_{1,j}^\mathrm{un}}$ which agrees with the minimal polynomial $P_{1,j}$ of $t_{1,j}$ modulo $(t_0')^{n_0}$ is Eisenstein, and has a root in $\mathcal{O}_{F_{1,j}^s}$ which agrees with $t_{1,j}$ modulo $t_{1,j}^{n_1}$. Applying the lemma again shows the existence of some $n\geq 2$ such that any polynomial $Q_0$ with coefficients in $\mathcal{O}_{F_0^\mathrm{un}}$ which agrees with $P_0$ modulo $t^{n}$ is Eisenstein, and has a root in $\mathcal{O}_{F_0^s}$ which agrees with $t_0$ modulo $t_0^{n_0}$. Now choose a $k$-algebra homomorphism $\lambda:R\rightarrow \mathcal{O}_F$ as provided by Lemma \ref{lemma: approx}: that is, factoring through the local ring of some smooth point on a curve inside $\spec{R}$ and agreeing with $\iota$ modulo $t^n$.

Since $\lambda$ is a $k$-algebra homomorphism, we have a canonical isomorphism $R_0^\mathrm{un}\otimes_{R,\lambda} \mathcal{O}_F \isomto \mathcal{O}_{F_0^\mathrm{un}}$, which therefore induces a homomorphism
\[ \lambda_0^\mathrm{un} :R_0^\mathrm{un} \rightarrow \mathcal{O}_{F_0^\mathrm{un}}\]
extending $\lambda$ and which agrees with $\iota_0^\mathrm{un}$ modulo $t^n$. Now let $Q_0=\lambda_0^\mathrm{un}(P_0)$ denote the image under $\lambda_0^\mathrm{un}$ of the minimal polynomial $P_0$ of $t_0$, this is therefore a monic polynomial with coefficients in $\mathcal{O}_{F_0^\mathrm{un}}$, which agrees with $P_0$ modulo $t^n$. Thus it is also Eisenstein, and by the choice of $n$ we know that $\mathcal{O}_{F_0^s}$ contains a root of $\lambda_0^\mathrm{un}(P_0)$ which is congruent to $t_0$ modulo $t_0^{n_0}$ and generates $\mathcal{O}_{F_0^s}$ as an $\mathcal{O}_{F_0^\mathrm{un}}$-algebra. This then allows us to extend $\lambda_0^\mathrm{un}$ to a homomorphism 
\[ \lambda_0^s: R_0^s \rightarrow \mathcal{O}_{F_0^s} \]
which agrees with $\iota_0^s$ modulo $t_0^{n_0}$, and since $\lambda_0^s(t_0)$ generates $\mathcal{O}_{F_0^s}$ as an $\mathcal{O}_{F_0^\mathrm{un}}$-algebra, we deduce that the diagram
\[ \xymatrix{ R_0^s \ar[r]^{\lambda_0^s} & \mathcal{O}_{F^s_0} \\ R \ar[r]^\lambda\ar[u] & \mathcal{O}_F \ar[u]  } \]
is coCartesian. We can then extend this to a homomorphism
\[ \lambda_0: R_0 \rightarrow \mathcal{O}_{F_0} \]
agreeing with $\iota_0$ modulo $(t_0')^{n_0}$, and forming a similar coCartesian diagram to $\lambda_0^s$. We now play exactly the same game for all of the $R_{1,j}$, to produce $\lambda_{1,j}: R_{1,j} \rightarrow \mathcal{O}_{F_{1,j}}$ extending all other $\lambda_0$ and all previous $\lambda_{1,j}^\#$, which agree with $\iota_{1,j}$ modulo $(t_{1,j}')^{n_{1,j}}$, and which form coCartesian diagrams
\[ \xymatrix{ R_{1,j} \ar[r]^{\lambda_{1,j}} & \mathcal{O}_{F_{1,j}} \\ R \ar[r]^\lambda \ar[u] & \mathcal{O}_F.  \ar[u] } \]
Now let $\mathcal{Z}$ be the base change of $\mathcal{Y}$ to $\mathcal{O}_F$ via $\lambda$, note that the generic fibre $\mathcal{Z}_F$ is globally defined by construction. Similarly define $\mathcal{Z}_0$ to be the base change of $\mathcal{Y}_0$ to $\mathcal{O}_{F_0}$ via $\lambda_0$, $\mathcal{Z}_{1,j}$ the base change of $\mathcal{Y}_{1,j}$ to $\mathcal{O}_{F_{1,j}}$ via $\lambda_{1,j}$, and $\mathcal{Z}_1:=\coprod_j \mathcal{Z}_{1,j}$, so we have a 2-truncated augmented simplicial scheme
\[ \mathcal{Z}_1 \rightrightarrows \mathcal{Z}_0 \rightarrow \mathcal{Z} \]
over $\mathcal{O}_F$, which gives a proper hypercover after base changing to $F$. For any $m\geq 2$ we can therefore take $n_1\geq m \max_j\{[F_{1,j}:F]\}$ to ensure that:
\begin{itemize}
\item $\mathcal{Z}_0$ is a proper and strictly semistable scheme over $\mathcal{O}_{F_0}$, and each $\mathcal{Z}_{1,j}$ is a proper and strictly semistable scheme over $\mathcal{O}_{F_{1,j}}$
\item there is an isomorphism
\[ \left[\mathcal{X}_1 \rightrightarrows \mathcal{X}_0\right] \otimes_{\mathcal{O}_F} \mathcal{O}_F/ t^m \isomto \left[ \mathcal{Z}_1 \rightrightarrows \mathcal{Z}_0 \right]\otimes_{\mathcal{O}_F} \mathcal{O}_F/ t^m  \]
of 2-truncated simplicial schemes, such that
\[ \mathcal{X}_0\otimes \mathcal{O}_F/ t^m \isomto \mathcal{Z}_0\otimes  \mathcal{O}_F/ t^m\]
is in fact an isomorphism of $\mathcal{O}_{F_0}/(t^m)$-schemes, and 
\[ \mathcal{X}_1\otimes \mathcal{O}_F/ t^m \isomto \mathcal{Z}_1\otimes  \mathcal{O}_F/ t^m\]
is obtained as a disjoint union of isomorphisms
\[ \mathcal{X}_{1,j}\otimes \mathcal{O}_F/ t^m \isomto \mathcal{Z}_{1,j}\otimes  \mathcal{O}_F/ t^m\]
of $\mathcal{O}_{F_{1,j}}/(t^m)$-schemes
\end{itemize}
Thus if we let $\mathcal{X}_{0,s}^\times$ and $\mathcal{Z}_{0,s}^\times$ denote the log schemes over $k_0^\times$ given by the special fibres of $\mathcal{X}_0$ and $\mathcal{Z}_0$, and $\mathcal{X}_{1,s}^\times$ and $\mathcal{Z}_{1,s}^\times$ the log schemes over $\coprod_{j=1}^s \spec{k^\times_{1,j}}$ given by the special fibres of $\mathcal{X}_1$ and $\mathcal{Z}_1$, then by Proposition \ref{prop: mod 2 functorial} there is an isomorphism
\[ [\mathcal{Z}^\times_{1,s} \rightrightarrows \mathcal{Z}^\times_{0,s}] \cong [\mathcal{X}^\times_{1,s} \rightrightarrows \mathcal{X}^\times_{0,s}]    \]
of 2-truncated simplicial log schemes over $k^\times$. Now by Propositions \ref{llc} and \ref{plc} there are isomorphisms
\begin{align*} H^i_\ell(\mathcal{X}_{0,F_0}) &\cong H^i_\ell(\mathcal{Z}_{0,F_0})  \\
H^i_\ell(\mathcal{X}_{1,F_{1,j}}) &\cong H^i_\ell(\mathcal{Z}_{{1,j},F_{1,j}})
\end{align*}
between the cohomology of the generic fibres of $\mathcal{X}_0,\mathcal{X}_{1,j}$ and $\mathcal{Z}_0,\mathcal{Z}_{1,j}$, as Weil--Deligne representations over $F_0$ and $F_{1,j}$ respectively. If we define the category
\[ \mathrm{Rep}_{\Q_{\ell}'}(\mathrm{WD}_{F_1}):= \prod_{j=1}^s  \mathrm{Rep}_{\Q_{\ell}'}(\mathrm{WD}_{F_{1,j}}) \]
of Weil--Deligne representations over $F_1:=\prod_j F_{1,j}$ to be the product of the categories of Weil Deligne representations over each $F_{1,j}$, then by Propositions \ref{prop: comp functorial 1} and \ref{prop: comp functorial 2}, the diagram
\[ \xymatrix{  H^i_\ell(\mathcal{X}_{0,F_0}) \ar[r] \ar[d]_\cong & H^i_\ell(\mathcal{X}_{1,F_1}) \ar[d]^\cong \\ H^i_\ell(\mathcal{Z}_{0,F_0}) \ar[r] & H^i_\ell(\mathcal{Z}_{1,F_1})} \]
(with horizontal arrows given by the differences of the two pullback maps) commutes via the restriction functor from Weil--Deligne representations over $F_0$ to Weil--Deligne representations over $F_1$.

Let $\mathrm{Ind}_{F_i}^F$ denote a right adjoint to the restriction functor from Weil--Deligne representations over $F$ to those over $F_i$: on the separable part this is the normal induction of representations, on the inseparable part it is a quasi-inverse to Frobenius pull-back, and $\mathrm{Ind}_{F_1}^F=\bigoplus_j \mathrm{Ind}^F_{F_{1,j}}$. We therefore have a commutative diagram
\[ \xymatrix{  \mathrm{Ind}_{F_0}^F H^i_\ell(\mathcal{X}_{0,F_0}) \ar[r] \ar[d]_\cong & \mathrm{Ind}_{F_1}^FH^i_\ell(\mathcal{X}_{1,F_1}) \ar[d]^\cong \\ \mathrm{Ind}_{F_0}^FH^i_\ell(\mathcal{Z}_{0,F_0}) \ar[r] & \mathrm{Ind}_{F_1}^FH^i_\ell(\mathcal{Z}_{1,F_1}), } \]
in particular, the kernels of both horizontal maps are isomorphic as Weil--Deligne representations over $F$. The proof of Theorem \ref{theo: fix} now boils down to the following claim.

\begin{proposition} \label{prop: kernel} Let $X_1\rightrightarrows X_0 \rightarrow X$ be a 2-truncated semisimplicial proper hypercover of a smooth and proper $F$-variety $X$, such that there exist finite field extensions $F_0/F$ and $F_{1,j}/F_0$ for $1\leq j\leq s$, with $X_0$ smooth over $F_0$, and $X_1=\coprod_j X_{1,j}$ with $X_{1,j}$ smooth over $F_{1,j}$. If we set $F_1=\prod_{j=1}^s F_{1,j}$, then 
\[ H^i_\ell(X) \cong \ker \left( \mathrm{Ind}_{F_0}^F H^i_\ell(X_0) \rightarrow \mathrm{Ind}_{F_1}^F H^i_\ell(X_1) \right)
\]
for all primes $\ell$.
\end{proposition}

\begin{proof} By taking $\widetilde{F}_1/F$ a sufficiently large finite extension such that all of the $F_{1,j}$ embed into $\widetilde{F}_1$ and applying \cite[Theorem 4.1]{dJ96}, we can extend $X_1\rightrightarrows X_0 \rightarrow X$ to a full proper hypercover $X_\bu\rightarrow X$ such that for $n\geq 2$ there exists a finite extension $F_n/\widetilde{F}_1$ with $X_n$ smooth over $F_n$. Now applying Lemma \ref{ind} we can see that the terms in $i$th column of the resulting spectral sequence have to be `quasi-pure' of weight $i$. Therefore the spectral sequence degenerates exactly as in the proof of Theorem \ref{main22}, and the proposition follows.
\end{proof}

We now deduce from the proposition that $H^i_\ell(X)\cong H^i_\ell(\mathcal{Z}_F)$ as Weil--Deligne representations for all $i,\ell$, and by construction $\mathcal{Z}_F$ is globally defined. This completes the proof of Theorem \ref{theo: fix}

\begin{remark} Note the use of the finite field hypothesis (via a weight argument) in the proof of Proposition \ref{prop: kernel}. It might be possible to relax the assumption to $k$ perfect using a more sophisticated argument.
\end{remark}

\end{document}